\documentclass{article}
\usepackage{graphicx}
\usepackage{cite}
\IfFileExists{authblk.sty}{\usepackage{authblk}}{}
\usepackage{amsthm}
\usepackage{amsmath}
\usepackage{amssymb}
\usepackage{mathrsfs}
\usepackage{enumitem}
\usepackage[left=3cm, right=3cm, top=1.0in, bottom=1.0in]{geometry}
\usepackage[colorlinks=true]{hyperref}

\allowdisplaybreaks[4]
\newtheorem{theorem}{Theorem}[section]
\newtheorem{lemma}[theorem]{Lemma}
\newtheorem{proposition}[theorem]{Proposition}

\newtheorem{definition}[theorem]{Definition}
\newtheorem{remark}[theorem]{Remark}

\newcommand{\R}{\mathbb R}

\newcommand{\Sym}{\operatorname{Sym}}

\newcommand{\tr}{\operatorname{tr}}
\newcommand{\rank}{\operatorname{rank}}
\newcommand{\diag}{\operatorname{diag}}

\newcommand{\dd}{\,\mathrm d}

\numberwithin{equation}{section}

\date{}
\title{Strict Convexity and Sharp Power Concavity for a Graphical
$\sigma_2$-Curvature Equation}
\author{Shuning Xu}
\begin{document}

\maketitle
\begin{abstract}
We prove strict convexity of the square-root transformation
\(v=-\sqrt{-u}\) for admissible solutions of a graphical
\(\sigma_2\)-curvature Dirichlet problem on smooth uniformly convex
domains. A key ingredient is a constant-rank theorem for \(D^2v\),
proved by a direct Ma--Xu type argument in dimension three and by
the Bian--Guan microscopic convexity principle together with
inverse-convexity methods in arbitrary dimensions. Combined with
boundary strict convexity and a domain-deformation argument, the
constant-rank theorem yields \(D^2v>0\) throughout the domain.
\end{abstract}

\noindent\textbf{Keywords:}
strict convexity; power concavity; graphical \(\sigma_2\)-curvature equation;
constant rank theorem; inverse convexity; optimal exponent.

\section{Introduction}
Convexity and concavity properties of solutions to elliptic equations
form a classical theme in nonlinear analysis. A central question is to
identify a transformation $\Phi(u)$ of the solution which enjoys a
global convexity or concavity property in a convex domain. One of the
earliest sharp results is the theorem of Makar--Limanov \cite{ML71},
who proved the square-root concavity of the torsion function in planar
convex domains. Brascamp--Lieb \cite{BrascampLieb1976} subsequently
established the log-concavity of the first Dirichlet eigenfunction on
convex domains. Caffarelli--Spruck \cite{CS82} developed PDE methods
for proving convexity properties of solutions to classical variational
problems, while Korevaar \cite{Korevaar1983} introduced a general
concavity maximum principle for nonlinear elliptic and parabolic
equations. Kennington \cite{Ken85} further developed this approach into
a systematic theory of power concavity for semilinear Dirichlet
problems. These results established the viewpoint that the natural
convexity of a solution is often revealed only after a nonlinear
transformation such as a logarithm or a fractional power.

A complementary route to strict convexity is to combine an initial
weak convexity property with a rigidity principle that prevents
degeneration of the Hessian in the interior. Constant-rank theorems
provide precisely such a mechanism and, when combined with deformation
arguments, allow weak convexity to be upgraded to strict convexity.
Caffarelli--Friedman \cite{CF85} developed this approach for semilinear
elliptic equations in dimension two, using a deformation argument
together with the strong maximum principle and a constant-rank property
of the Hessian to obtain convexity results. Korevaar--Lewis \cite{KL87}
subsequently extended the constant-rank theory to higher dimensions and
to a broader class of elliptic equations, further establishing
constant-rank rigidity as an effective tool in convexity problems.
From a different direction, Alvarez--Lasry--Lions \cite{ALL97}
developed a viscosity and convex-envelope approach to convexity for
fully nonlinear equations and introduced an inverse-matrix convexity
condition that later became fundamental in the structural theory of
constant-rank and microscopic convexity principles.

For fully nonlinear geometric equations, the interaction between
convexity, constant rank, and deformation became particularly
effective. Guan--Ma \cite{GM03} developed this strategy for the
Christoffel--Minkowski Hessian equation. It was subsequently extended
to prescribed Weingarten curvature equations by Guan--Lin--Ma
\cite{GLM06} and to related admissible solutions by Guan--Ma--Zhou
\cite{GMZ06}. Caffarelli--Guan--Ma \cite{CGM07} established a general
constant-rank theorem for fully nonlinear elliptic equations and, in
particular, for the second fundamental form of convex hypersurfaces
satisfying prescribed curvature equations. For the three-dimensional
$\sigma_2$-Hessian equation, Ma--Xu \cite{MX08} obtained a direct
constant-rank argument and combined it with a deformation procedure to
prove convexity. Bian--Guan \cite{BG09,BG10} later formulated the
microscopic convexity principle under a general structural
inverse-convexity condition, providing a unified framework for many
such constant-rank results.

More recently, Li–Ma–Salani \cite{LMS26} developed a
hyperbolic-polynomial approach to the inverse-convexity structure
arising in the $\sigma_2$ problem. Their result provides an
algebraic mechanism that allows the Bian–Guan principle to be applied
without the dimension-three third-order calculation.  These developments provide the
structural tools for the constant-rank argument used
below.

The graphical $\sigma_2$-curvature equation considered here lies beyond the standard Hessian setting, since the curvature operator depends nonlinearly on both $D^2u$ and $Du$. Consequently, the existing constant-rank and convexity results for Hessian equations do not apply directly. Our goal is to establish global strict convexity for the natural square-root transformation and to determine the corresponding optimal power-concavity exponent.

A geometric model particularly relevant to the present work is the
mean-curvature-type equation
$H[u]=\omega^{-3}$, where
$H[u]=\operatorname{div}(Du/\sqrt{1+|Du|^2})$ and
$\omega=\sqrt{1+|Du|^2}$. The corresponding constant-rank phenomenon belongs
to the classical framework of Korevaar--Lewis and its later extensions by
Bian--Guan \cite{KL87,BG09,BG10}. This equation also has a natural geometric
interpretation. If
$\nu=(-Du,1)/\omega$ is the upward unit normal to the graph, then
$\omega^{-1}=\langle\nu,e_{n+1}\rangle$ is its vertical angle function, and
hence the equation can be written, up to the choice of orientation, as
$H=\langle\nu,e_{n+1}\rangle^3$. Translating graphs for flows by powers of
mean curvature satisfy an equation of this type; in particular, the exponent
$-3$ corresponds to the translating equation for the $H^{1/3}$-flow
\cite{JJ11}. This mean-curvature model motivates the question considered
below: whether an analogous constant-rank structure persists when the first
elementary symmetric curvature $\sigma_1(\kappa)=H$ is replaced by the
fully nonlinear curvature $\sigma_2(\kappa)$.

Let $\Omega\subset\R^n$ and let
\[
\Sigma_u=\{(x,u(x)):x\in\Omega\}\subset\R^{n+1}
\]
be the graph of a smooth function. Set
\[
\omega=\sqrt{1+|Du|^2}.
\]
Let $\kappa[u]=(\kappa_1[u],\ldots,\kappa_n[u])$ denote the principal curvatures of the graph. We study
\begin{equation}\label{eq:real-main}
\sigma_2(\kappa[u])=(1+|Du|^2)^{-2}=\omega^{-4},\,u<0 \,in\,\Omega,
\, u=0 \, on\,\partial\Omega.
\end{equation}
The graph is called $\Gamma_2$-admissible when its principal-curvature vector belongs to the G\aa rding cone $\Gamma_2$ at every point.

Our square-root transformation is
\begin{equation}\label{eq:transform-intro}
v=-\sqrt{-u}<0,\qquad u=-v^2.
\end{equation}
The exponent $-4$ is distinguished by this transformation. For the more general family
\[
\sigma_2(\kappa[u])=\omega^b,
\]
we shall prove that the transformed equation has right-hand side
\[
\frac14\bigl(1+4v^2|Dv|^2\bigr)^{(b+4)/2}.
\]
Thus $b=-4$ is precisely the exponent for which the right-hand side becomes constant. This parallels the distinguished exponent $-3$ in the mean-curvature equation, and the constant-rank results below may be viewed as a fully nonlinear $\sigma_2$ extension of the corresponding mean-curvature-type constant-rank theory \cite{KL87,BG09,BG10}; see Remark~\ref{rem:mean-curvature}.

Our main result is a strict-convexity theorem for the square-root transformation \(v=-\sqrt{-u}\).
\begin{theorem}\label{thm:real-strict-convexity}
Let \(n\ge 3\), and let \(\Omega\subset\mathbb R^n\) be a bounded,
smooth, uniformly convex domain. Let
$u\in C^\infty(\Omega)\cap C^{1,1}(\overline{\Omega})$
be such that its graph is \(\Gamma_2\)-admissible in \(\Omega\),
\(u\) satisfies \eqref{eq:real-main} in \(\Omega\), and
$u=0\,\,\text{on }\partial\Omega.$
Set   
$v=-\sqrt{-u},$
then
$u<0\,\,\text{in }\Omega,$
and \(v\) is strictly convex in \(\Omega\).
\end{theorem}

The proof of Theorem~\ref{thm:real-strict-convexity} combines a new
constant-rank argument with a global deformation procedure. Compared
with the standard $\sigma_2$-Hessian problem of Ma--Xu \cite{MX08},
the transformed graphical equation contains additional
gradient-dependent Newton-tensor terms. In dimension three these terms can still be handled by a Ma--Xu type reduction to a positive quadratic form, whereas in higher dimensions we use the Bian--Guan microscopic convexity principle. The latter requires extending the inverse-convexity structure used by Li--Ma--Salani \cite{LMS26} to accommodate the additional graphical $T_2$-term.

The constant-rank theorem provides only the interior rigidity needed
for the global argument. We first obtain strict convexity near the
boundary and construct a strictly convex radial solution on the unit
ball, and then deform the ball to the prescribed domain through a
Minkowski family with uniform a priori estimates. At the closedness
step, the constant-rank theorem prevents degeneration of the limiting weakly convex solution, while boundary strict convexity forces the Hessian to have full rank. This yields
$D^2(-\sqrt{-u})>0$ throughout the domain.

The remainder of the paper is organized as follows. Section~\ref{sec:prelim} collects the basic definitions and structural tools used in the proofs. In Section~\ref{sec:proofs}, we derive the transformed equations and prove the constant-rank results in dimension three and in arbitrary dimensions, and we also discuss the corresponding mean-curvature equation and complete the domain-deformation proof of strict convexity. In Section~\ref{sec:square-root-optimality}, we construct counterexamples to show the power concavity exponent is optimal.

\section*{Acknowledgments}
The author thanks Professor Xi-Nan Ma for bringing this question to her attention. This work was supported by the National Natural Science Foundation of China [grant number 2025YFA1017601].

\section{Preliminaries and structural tools}\label{sec:prelim}

\subsection{Elementary symmetric functions and Newton tensors}

\begin{definition}
\label{def:sigma-garding}
Let $M$ be a real symmetric $N\times N$ matrix with eigenvalues
$\lambda(M)=(\lambda_1,\ldots,\lambda_N)$. For $0\le k\le N$, the $k$-th
elementary symmetric function is defined by
\begin{equation}\label{eq:sigmak}
\sigma_0(M)=1,
\qquad
\sigma_k(M)
=
\sum_{1\le i_1<\cdots<i_k\le N}
\lambda_{i_1}\cdots\lambda_{i_k}.
\end{equation}
For $1\le k\le N$, the $k$-th G\aa rding cone is
\begin{equation}\label{eq:Gammak}
\Gamma_k
=
\left\{
\lambda\in\R^N:
\sigma_j(\lambda)>0,\quad 1\le j\le k
\right\}.
\end{equation}
We write $M\in\Gamma_k$ if $\lambda(M)\in\Gamma_k$.

In particular,
\begin{equation}\label{eq:sigma2trace}
\sigma_1(M)=\tr M,
\qquad
\sigma_2(M)
=
\frac12\left((\tr M)^2-\tr(M^2)\right),
\end{equation}
and
\begin{equation}\label{eq:Gamma2}
\Gamma_2
=
\left\{
\lambda\in\R^N:
\sigma_1(\lambda)>0,\ 
\sigma_2(\lambda)>0
\right\}.
\end{equation}
\end{definition}

\begin{definition}
\label{def:newton-tensors}
For $0\le k\le N$, the $k$-th Newton tensor associated with $M$ is
\begin{equation}\label{eq:newton-general}
T_k(M)
=
\sigma_k(M)I-\sigma_{k-1}(M)M+\cdots+(-1)^kM^k,
\end{equation}
where $T_0(M)=I$. Equivalently,
\begin{equation}\label{eq:newton-recursion}
T_k(M)=\sigma_k(M)I-MT_{k-1}(M).
\end{equation}
The Newton tensors satisfy the differential identity
\begin{equation}\label{eq:newton-derivative}
D\sigma_{k+1}(M)[H]
=
\tr\bigl(T_k(M)H\bigr).
\end{equation}
See, for example, Reilly \cite{Reilly77} for the Newton-transformation
formalism.

In the cases used below,
\begin{equation}\label{eq:T1T2}
T_1(M)=(\tr M)I-M,
\qquad
T_2(M)=\sigma_2(M)I-(\tr M)M+M^2.
\end{equation}
Hence, for a real vector $q$,
\begin{align}
T_1(M)[q,q]
&=
|q|^2\tr M-q^*Mq,
\label{eq:T1-explicit-main}\\
T_2(M)[q,q]
&=
\sigma_2(M)|q|^2
-(\tr M)q^*Mq
+q^*M^2q.
\label{eq:T2-explicit-main}
\end{align}
For $M\in\Gamma_k$, one has
\begin{equation}\label{eq:newton-positive}
T_{k-1}(M)>0.
\end{equation}
In particular,
\begin{equation}\label{eq:T1-positive}
M\in\Gamma_2
\quad\Longrightarrow\quad
T_1(M)>0.
\end{equation}
These are standard properties of the $k$-Hessian operator; see
\cite{CNS85}.
\end{definition}

\begin{lemma}
\label{lem:rank-one-perturbation}
Let $M$ be a real symmetric $N\times N$ matrix and let $q$
be a real vector. Then, for $1\le k\le N$ and every $t\in\R$,
\begin{equation}\label{eq:rankone-general}
\sigma_k(M+tqq^{T})
=
\sigma_k(M)
+t\,T_{k-1}(M)[q,q].
\end{equation}
In particular,
\begin{equation}\label{eq:rankone1}
\sigma_2(M+tqq^{T})
=
\sigma_2(M)
+tT_1(M)[q,q].
\end{equation}
Moreover,
\begin{equation}\label{eq:rankone2}
T_2(M+tqq^{T})[q,q]
=
T_2(M)[q,q].
\end{equation}
The $\sigma_2$ identity \eqref{eq:rankone1} is also used explicitly in
Chen--Li--Ma \cite[Eq.~(3.2)]{CLM26}.
\end{lemma}

\begin{proof}
Using the matrix determinant lemma in adjugate form,
\begin{align}
\det\bigl(I+s(M+tqq^*)\bigr)
&=
\det(I+sM)
+st\,q^*\operatorname{adj}(I+sM)q.
\label{eq:rankone-det}
\end{align}
On the other hand,
\begin{equation}\label{eq:det-generating}
\det(I+sM)
=
\sum_{j=0}^N \sigma_j(M)s^j,
\end{equation}
while the Newton expansion of the adjugate is
\begin{equation}\label{eq:adj-newton}
\operatorname{adj}(I+sM)
=
\sum_{j=0}^{N-1}T_j(M)s^j.
\end{equation}
Comparing the coefficient of $s^k$ in
\eqref{eq:rankone-det} gives
\[
\sigma_k(M+tqq^*)
=
\sigma_k(M)
+tT_{k-1}(M)[q,q],
\]
which proves \eqref{eq:rankone-general}.

For $N\ge3$, applying \eqref{eq:rankone-general} to $\sigma_3$ gives
\[
\sigma_3(M+tqq^*)
=
\sigma_3(M)
+tT_2(M)[q,q].
\]
Replacing $M$ by $M+tqq^*$ and comparing the slope in the same rank-one
direction yields
\[
T_2(M+tqq^*)[q,q]
=
T_2(M)[q,q].
\]
If $N=2$, the same identity is automatic since $T_2(M)\equiv0$ by the
Cayley--Hamilton theorem.
\end{proof}

\subsection{Hyperbolic polynomials and inverse convexity}

We next collect the hyperbolic-polynomial and inverse-convexity tools
used in the constant-rank arguments. The hyperbolic-polynomial results
below provide the algebraic mechanism behind the inverse-convexity
theorems that will be used in the constant-rank argument below.

\begin{definition}\label{def:hyperbolic}
Let $V$ be a finite-dimensional real vector space and let $p$ be a homogeneous polynomial. The polynomial $p$ is \emph{hyperbolic with respect to} $e\in V$ if $p(e)>0$ and, for every $x\in V$, the polynomial
\[
t\longmapsto p(x+te)
\]
has only real roots. The connected component of $\{p>0\}$ containing $e$ is the hyperbolicity cone of $p$.
\end{definition}

The determinant on the real symmetric matrix space is hyperbolic with respect to the identity, and its hyperbolicity cone is the positive definite cone. We write
\[
D_Ep(x)=\left.\frac{\dd}{\dd t}\right|_{t=0}p(x+tE).
\]
We use three standard facts; see \cite{Garding59,BGLS01,Renegar06}.

\begin{theorem}\label{thm:garding}
Let $p$ be hyperbolic and let $E$ lie in its hyperbolicity cone. Then
\[
x\longmapsto\frac{p(x)}{D_Ep(x)}
\]
is concave on the hyperbolicity cone.
\end{theorem}

\begin{theorem}
\label{thm:derivhyper}
Let $p$ be a homogeneous polynomial of degree $d$ which is
hyperbolic with respect to $e$, and let $\Gamma_p$ denote its
hyperbolicity cone. If $E\in \Gamma_p$, then the directional
derivative
\[
D_Ep(x)
=
\left.\frac{d}{dt}\right|_{t=0}p(x+tE)
\]
is a homogeneous hyperbolic polynomial of degree $d-1$.
Moreover, $D_Ep$ may be taken to be hyperbolic with respect to
$e$, and its hyperbolicity cone contains the original one:
\[
\Gamma_p\subset \Gamma_{D_Ep}.
\]

More generally, if $E\in\overline{\Gamma_p}$ and
$D_Ep\not\equiv0$, the same conclusion holds by approximation:
one may take
\[
E_\varepsilon=E+\varepsilon e\in\Gamma_p,
\qquad \varepsilon>0,
\]
apply the preceding statement to $D_{E_\varepsilon}p$, and then
let $\varepsilon\downarrow0$.
\end{theorem}

\begin{remark}
\label{rem:boundary-hyperbolic-direction}
The approximation in Theorem~\ref{thm:derivhyper}
is particularly useful for the determinant polynomial. Indeed,
$\det$ is hyperbolic on the positive definite cone, while
positive semidefinite matrices occur naturally as differentiation
directions. If $E\geq0$ and $E\neq0$, then
\[
E+\varepsilon I>0,
\]
and hence $D_{E+\varepsilon I}\det$ is hyperbolic. Passing to the
limit gives the corresponding statement for $D_E\det$.
See \cite{BGLS01,Renegar06} for the general hyperbolic-polynomial
framework.
\end{remark}

\begin{theorem}\label{thm:logderivative}
Let $p$ be positive and hyperbolic on its cone and let $E$ lie in the cone. Then
\[
x\longmapsto\frac{D_Ep(x)}{p(x)}
\]
is convex. The same non-strict convexity conclusion remains valid for $E\in\overline{\Gamma_p}$, provided $D_Ep\not\equiv0$, by
approximating $E$ with interior directions.
\end{theorem}

The last statement is the application form of Corollary 4.7 of Bauschke--G\"uler--Lewis--Sendov \cite{BGLS01}; it is also the form used in \cite{LMS26}.

Fix a unit vector $\alpha\in\R^n$ and set
\begin{equation}\label{eq:Qreal}
Q=I-\alpha\otimes\alpha.
\end{equation}
We use the following theorem.

\begin{theorem}\label{thm:LMS}
For every $\lambda>0$, the function
\begin{equation}\label{eq:LMSquotient}
A\longmapsto
\frac{\sigma_2(A^{-1})-\lambda}{\tr(QA^{-1})}
\end{equation}
is convex on $\Sym_{++}(n)$.
\end{theorem}

The hyperbolic-polynomial representation behind Theorem~\ref{thm:LMS} is worth recording because the same polynomial will control an additional term in our equation. Put
\[
b_\alpha=\alpha\otimes\alpha,
\qquad
R_0=\frac12(I+b_\alpha),
\qquad
g(A)=D_Q\det A.
\]
Then
\begin{equation}\label{eq:g-real}
g(A)=\det A\,\tr(QA^{-1}),
\end{equation}
and
\begin{equation}\label{eq:D_R_g-real}
D_{R_0}g(A)=\det A\,\sigma_2(A^{-1}).
\end{equation}
Thus
\[
\frac{\tr(QA^{-1})}{\sigma_2(A^{-1})}
=\frac{g(A)}{D_{R_0}g(A)}.
\]
By G\aa rding quotient concavity the last quotient is concave. The remaining constant term in \eqref{eq:LMSquotient} is handled using the reciprocal-trace concavity of Alvarez--Lasry--Lions \cite{ALL97}.

In the transformed equation, the preceding inverse-convexity term is
coupled with an additional scalar variable. To preserve convexity under
the rescaling used later, we shall use the following elementary
perspective principle.

\begin{lemma}\label{lem:perspective}
Let $\Theta$ be convex on a positive definite cone and homogeneous of degree $-1$:
\[
\Theta(tB)=t^{-1}\Theta(B),\qquad t>0.
\]
Then
\[
(B,a)\longmapsto a^2\Theta(B),\qquad a>0,
\]
is jointly convex.
\end{lemma}

\begin{proof}
Define
\[
\Phi(B,a)=a^2\Theta(B),\qquad a>0.
\]
Since $\Theta$ is homogeneous of degree $-1$,
\begin{equation}\label{eq:perspective-identity}
\Phi(B,a)
=
a\,\Theta(B/a).
\end{equation}

Let $(B_1,a_1)$ and $(B_2,a_2)$ be two points in the domain, and let
$0\le\theta\le1$. Set
\begin{equation}\label{eq:perspective-combination}
B_\theta
=
\theta B_1+(1-\theta)B_2,
\qquad
a_\theta
=
\theta a_1+(1-\theta)a_2.
\end{equation}
Since $a_1,a_2>0$, also $a_\theta>0$. Define
\begin{equation}\label{eq:perspective-weights}
\mu_1
=
\frac{\theta a_1}{a_\theta},
\qquad
\mu_2
=
\frac{(1-\theta)a_2}{a_\theta}.
\end{equation}
Then $\mu_1,\mu_2\ge0$, $\mu_1+\mu_2=1$, and
\begin{equation}\label{eq:perspective-scaled-combination}
\frac{B_\theta}{a_\theta}
=
\mu_1\frac{B_1}{a_1}
+
\mu_2\frac{B_2}{a_2}.
\end{equation}
By the convexity of $\Theta$,
\begin{align}
\Theta\left(\frac{B_\theta}{a_\theta}\right)
&\le
\mu_1\Theta\left(\frac{B_1}{a_1}\right)
+
\mu_2\Theta\left(\frac{B_2}{a_2}\right).
\end{align}
Multiplying by $a_\theta$ and using
\eqref{eq:perspective-weights}, we obtain
\begin{align}
a_\theta
\Theta\left(\frac{B_\theta}{a_\theta}\right)
&\le
\theta a_1
\Theta\left(\frac{B_1}{a_1}\right)
+
(1-\theta)a_2
\Theta\left(\frac{B_2}{a_2}\right).
\end{align}
Finally, by the homogeneity of $\Theta$,
\[
a\,\Theta(B/a)=a^2\Theta(B).
\]
Hence
\[
\Phi(B_\theta,a_\theta)
\le
\theta\Phi(B_1,a_1)
+
(1-\theta)\Phi(B_2,a_2),
\]
which proves that $(B,a)\mapsto a^2\Theta(B)$ is jointly convex.
\end{proof}


We use the following application form of the microscopic convexity principle \cite{BG09,BG10}. The preceding convexity results will be used to verify the structural
condition in the microscopic convexity principle. We record below the
form of the Bian--Guan constant-rank theorem that will be applied in
the real graphical equation.

\begin{theorem}\label{thm:BG}
Let $D\subset\R^N$ be connected and let $v\in C^{3,1}(D)$ be convex. Suppose
\[
F(D^2v,Dv,v,x)=0,
\]
where $F\in C^{2,1}$. Assume:
\begin{enumerate}[label=\textup{(\roman*)}]
\item $F$ is elliptic along the solution, i.e.
\[
(F^{\alpha\beta})=\left(\frac{\partial F}{\partial r_{\alpha\beta}}\right)>0;
\]
\item $F(0,Dv,v,x)\neq0$ along the solution;
\item for every fixed relevant gradient $p$, the set
\begin{equation}\label{eq:BGlevel}
\{(A,z,x):A\in\Sym_{++}(N),\ F(A^{-1},p,z,x)\le0\}
\end{equation}
is locally convex near the relevant points.
\end{enumerate}
Then $\rank D^2v$ is constant in $D$.
\end{theorem}

%


\section{Proof of Theorem~\ref{thm:real-strict-convexity}}\label{sec:proofs}

\subsection{The constant-rank theorem: two proofs}\label{real-const-thm}

To prove Theorem~\ref{thm:real-strict-convexity}, we first need  the constant rank theorem. For $n=3$, we can adopt a Ma–Xu type proof from \cite{MX08}. For $n\geq 3$, we build on the inverse-convexity method developed in \cite{LMS26}. The additional graphical term involving the second Newton tensor can be handled by the same hyperbolic-polynomial structure underlying the inverse-convexity argument in \cite{LMS26}, together with a perspective argument.

\begin{theorem}\label{thm:realn}
Let $n\ge3$, let $\Omega\subset\R^n$ be connected, and let $u\in C^4(\Omega)$ satisfy $u<0$. Assume that the graph of $u$ is $\Gamma_2$-admissible and that \eqref{eq:real-main} holds. Set $v=-\sqrt{-u}$. If
$D^2v\ge0\,\,\text{in }\Omega$, then $\rank D^2v$ is constant in $\Omega$.
\end{theorem}

We first derive the exact equation satisfied by $v=-\sqrt{-u}$. This calculation also explains why the exponent $-4$ is singled out.

\begin{lemma}\label{lem:real-transform}
Let $u<0$, set $v=-\sqrt{-u}$, and let $\omega=\sqrt{1+|Du|^2}$. If
\begin{equation}\label{eq:b-family}
\sigma_2(\kappa[u])=\omega^b,
\end{equation}
then $v$ satisfies
\begin{align}
&v^2\sigma_2(D^2v)
+v\bigl(|Dv|^2\Delta v-v_iv_jv_{ij}\bigr)\notag\\
&\quad
+4v^4\bigl[
\sigma_2(D^2v)|Dv|^2
-(\Delta v)v_iv_jv_{ij}
+v_iv_{ik}v_{kj}v_j
\bigr]\notag\\
&\hspace{30mm}
=\frac14\bigl(1+4v^2|Dv|^2\bigr)^{(b+4)/2}.
\label{eq:v-general-b}
\end{align}
Equivalently,
\begin{align}\label{eq:v-general-Newton}
&v^2\sigma_2(D^2v)
+vT_1(D^2v)[Dv,Dv]
+4v^4T_2(D^2v)[Dv,Dv]\notag\\
&\hspace{35mm}=\frac14\bigl(1+4v^2|Dv|^2\bigr)^{(b+4)/2}.
\end{align}
\end{lemma}
\begin{proof}
Since $u=-v^2$,
\begin{equation}\label{eq:u-derivative}
Du=-2vDv,
\qquad
D^2u=-2vD^2v-2Dv\otimes Dv.
\end{equation}
Set
\begin{equation}
W=1+4v^2|Dv|^2=\omega^2,
\qquad
g=I+4v^2Dv\otimes Dv,
\qquad
K=(-v)D^2v-Dv\otimes Dv.
\end{equation}
Then $D^2u=2K$. For a graph $u$, the Weingarten map is similar to the
symmetric matrix
\begin{equation}
\frac{1}{\omega}
\bigl(I+Du\otimes Du\bigr)^{-1/2}
D^2u
\bigl(I+Du\otimes Du\bigr)^{-1/2}.
\end{equation}
Therefore, using \eqref{eq:u-derivative}, the symmetric representative
of the shape operator is
\begin{equation}\label{eq:real-shape-v}
\frac{2}{\sqrt{W}}\,g^{-1/2}Kg^{-1/2}.
\end{equation}

We next compute the $\sigma_2$ term. For any symmetric matrix $H$,
\begin{equation}\label{eq:3.8}
\det(I+t g^{-1}H)
=
\frac{\det(g+tH)}{\det g}.
\end{equation}
Since $g=I+4v^2Dv\otimes Dv$, the rank-one determinant formula gives
\begin{equation}
\det(g+tH)
=
\det(I+tH)
+
4v^2Dv^T\operatorname{adj}(I+tH)Dv.
\end{equation}
Since both sides of \eqref{eq:3.8} are polynomials in \(t\), we only need the
coefficient of $t^2$ for the computation of $\sigma_2$. In \(\det(I+tH)\), this coefficient is \(\sigma_2(H)\), whereas the Newton expansion of the adjugate shows that the contribution of the rank-one term to the same coefficient is \(4v^2T_2(H)[Dv,Dv]\). Since \(\det g=W\), comparison of the \(t^2\)-coefficients gives \eqref{eq:real-metric-sigma2}.
\begin{equation}\label{eq:real-metric-sigma2}
W\sigma_2(g^{-1}H)
=
\sigma_2(H)
+
4v^2T_2(H)[Dv,Dv].
\end{equation}

Taking $H=K$ and applying the rank-one perturbation formulas from
Lemma~\ref{lem:rank-one-perturbation}, together with the homogeneity of
the Newton tensors, gives
\begin{align}
\sigma_2(K)
&=
v^2\sigma_2(D^2v)
+
vT_1(D^2v)[Dv,Dv],
\label{eq:real-Ksigma2}\\
T_2(K)[Dv,Dv]
&=
v^2T_2(D^2v)[Dv,Dv].
\label{eq:real-KT2}
\end{align}
Since multiplication of a matrix by $2/\sqrt W$ multiplies
$\sigma_2$ by $4/W$, equations
\eqref{eq:real-shape-v}--\eqref{eq:real-KT2} yield
\begin{equation}\label{eq:real-sigma2-final}
\sigma_2(\kappa[u])
=
\frac{4}{W^2}
\left[
v^2\sigma_2(D^2v)
+
vT_1(D^2v)[Dv,Dv]
+
4v^4T_2(D^2v)[Dv,Dv]
\right].
\end{equation}
Substituting this identity into \eqref{eq:b-family} gives
\eqref{eq:v-general-Newton}, and expanding the Newton contractions gives
\eqref{eq:v-general-b}.
\end{proof}

For $b=-4$, the right-hand side becomes the constant $1/4$. We therefore set
\begin{equation}\label{eq:Freal}
F(R,p,z)
=z^2\sigma_2(R)+zT_1(R)[p,p]+4z^4T_2(R)[p,p]-\frac14.
\end{equation}
Then $v$ satisfies
\begin{equation}\label{eq:Freal-eqn}
F(D^2v,Dv,v)=0.
\end{equation}
In expanded form,
\begin{align}
0={}&v^2\sigma_2(D^2v)
+v\bigl(|Dv|^2\Delta v-v_iv_jv_{ij}\bigr)\notag\\
&+4v^4\bigl[
\sigma_2(D^2v)|Dv|^2-(\Delta v)v_iv_jv_{ij}+v_iv_{ik}v_{kj}v_j
\bigr]-\frac14.
\label{eq:Freal-expanded}
\end{align}

For later use we record ellipticity and exclude ranks zero and one.

\begin{lemma}\label{lem:real-ellipticity}
Assume the graph is $\Gamma_2$-admissible. Then the linearization of \eqref{eq:Freal-eqn} with respect to $D^2v$ is positive definite. If moreover $D^2v\ge0$, then $\rank D^2v\ge2$.
\end{lemma}

\begin{proof}
With $g$ and $K$ from the proof of Lemma~\ref{lem:real-transform},
\begin{equation}\label{eq:Freal-symmetric}
F(D^2v,Dv,v)+\frac14
=W\sigma_2(g^{-1/2}Kg^{-1/2}).
\end{equation}
The matrix $g^{-1/2}Kg^{-1/2}$ is a positive scalar multiple of the shape operator, hence lies in $\Gamma_2$. For a symmetric variation $\xi$ of $D^2v$,
\begin{align}
D_RF[\xi]
={}&W(-v)T_1(g^{-1/2}Kg^{-1/2}):
(g^{-1/2}\xi g^{-1/2}).
\label{eq:real-linearization}
\end{align}
Because $-v>0$ and $T_1>0$ on $\Gamma_2$, ellipticity follows.

If $D^2v\ge0$ has rank at most one, then $\sigma_2(D^2v)=0$, $T_2(D^2v)=0$, and $T_1(D^2v)[Dv,Dv]\ge0$. Since $v<0$, the left-hand side of
\[
v^2\sigma_2(D^2v)+vT_1(D^2v)[Dv,Dv]+4v^4T_2(D^2v)[Dv,Dv]=\frac14
\]
is nonpositive, a contradiction.
\end{proof}

Now we can prove Theorem~\ref{thm:realn}.

\medskip
\noindent\textbf{Method I: the direct argument in dimension three.}
We first give a direct proof in the three-dimensional case, following the
Ma--Xu approach \cite{MX08}. This method is specific to $n=3$ in its present
form: at a minimum-rank point the Hessian has rank two, and the differentiated
equation can be reduced to a two-variable quadratic form. The maximum
principle then yields the constant-rank conclusion.

By Lemma~\ref{lem:real-ellipticity}, only ranks two and three can occur. Suppose the minimum rank is two. At a minimum-rank point, rotate coordinates so that
\begin{equation}\label{eq:rank2}
D^2v=\diag(\lambda_1,\lambda_2,0),
\qquad \lambda_1,\lambda_2>0.
\end{equation}
Take
\begin{equation}\label{eq:Pdet}
P=\det D^2v.
\end{equation}
As in \cite{MX08}, write $X\sim Y$ when $X-Y$ is controlled by $C(P+|\nabla P|)$ in a minimum-rank neighborhood. Then
\begin{equation}\label{eq:Pderivatives}
P\sim\lambda_1\lambda_2v_{33},
\qquad
P_i\sim\lambda_1\lambda_2v_{33i},
\end{equation}
and
\begin{equation}\label{eq:Pij-main}
P_{ij}\sim
\lambda_1\lambda_2v_{33ij}
-2\lambda_1v_{23i}v_{23j}
-2\lambda_2v_{13i}v_{13j}.
\end{equation}

At \eqref{eq:rank2}, the equation reads
\begin{align}
&v^2\lambda_1\lambda_2
+v\bigl[
\lambda_2v_1^2+\lambda_1v_2^2+(\lambda_1+\lambda_2)v_3^2
\bigr]
+4v^4\lambda_1\lambda_2v_3^2
=\frac14.
\label{eq:rank2eq}
\end{align}
The Hessian derivatives of $F$ needed in the $x_1,x_2$ block are
\begin{align}
F^{11}
&=(-v)\left[(-v)\lambda_2(1+4v^2v_3^2)-(v_2^2+v_3^2)\right],
\label{eq:F11-main}\\
F^{22}
&=(-v)\left[(-v)\lambda_1(1+4v^2v_3^2)-(v_1^2+v_3^2)\right],
\label{eq:F22-main}\\
F^{12}&=-vv_1v_2.
\label{eq:F12-main}
\end{align}
The scalar derivative is
\begin{align}
F_v
={}&2v\lambda_1\lambda_2
+\lambda_2v_1^2+\lambda_1v_2^2+(\lambda_1+\lambda_2)v_3^2
+16v^3\lambda_1\lambda_2v_3^2.
\label{eq:Fv-main}
\end{align}

Set
\begin{equation}\label{eq:thirdvars}
\xi=v_{113},
\qquad
\eta=v_{223},
\qquad
\zeta=v_{123}.
\end{equation}
Since $v_{13}=v_{23}=v_{33}=0$ at the diagonal point, differentiating \eqref{eq:Freal-eqn} in the $x_3$ direction gives
\begin{equation}\label{eq:firstdiff-main}
F^{11}\xi+F^{22}\eta+2F^{12}\zeta+F_vv_3=0.
\end{equation}
Ellipticity gives $F^{22}>0$, hence
\begin{equation}\label{eq:eta-main}
\eta=-\frac{F^{11}\xi+2F^{12}\zeta+F_vv_3}{F^{22}}.
\end{equation}
This is the first elimination.

Differentiating once more in the $x_3$ direction and using $v_{i3}=0$ together with $v_{i33}=v_{33i}\sim0$, we obtain
\begin{equation}\label{eq:seconddiff-main}
F^{ij}v_{ij33}
+F_{RR}[v_{ij3},v_{kl3}]
+2v_3F_{Rv}[v_{ij3}]
+F_{vv}v_3^2\sim0.
\end{equation}
Let us explain which terms survive in the second differentiation. At the diagonal minimum-rank point we have
\(v_{13}=v_{23}=v_{33}=0\), while
\(v_{i33}=v_{33i}\sim0\) in the sense introduced above.
Therefore all terms containing a first variation of the gradient
variable \(Dv\), as well as the terms containing \(v_{33i}\), are
absorbed into \(C(P+|\nabla P|)\). Since \(F\) has no explicit
\(x\)-dependence, the only quadratic terms which remain at the
principal level are the second variation with respect to the Hessian,
the mixed Hessian--\(v\) variation, and the pure \(v\)-variation.
This gives precisely the three terms displayed in \eqref{eq:seconddiff-main}.
A direct differentiation of \eqref{eq:Freal} yields
\begin{equation}\label{eq:Frr-main}
F_{RR}[v_{ij3},v_{kl3}]
=2v^2(1+4v^2v_3^2)(\xi\eta-\zeta^2),
\end{equation}
\begin{align}
F_{Rv}[v_{ij3}]
={}&\bigl(v_2^2+v_3^2+2v\lambda_2+16v^3\lambda_2v_3^2\bigr)\xi
\notag\\
&+\bigl(v_1^2+v_3^2+2v\lambda_1+16v^3\lambda_1v_3^2\bigr)\eta
-2v_1v_2\zeta,
\label{eq:Frv-main}
\end{align}
and
\begin{equation}\label{eq:Fvv-main}
F_{vv}=2\lambda_1\lambda_2(1+24v^2v_3^2).
\end{equation}
We next substitute the twice differentiated equation into the
second derivative of the test function \(P=\det D^2v\).
By \eqref{eq:Pij-main}, the term containing \(v_{33ij}\) appears with the factor
\(\lambda_1\lambda_2\), and contracting it with the linearized
coefficients \(F^{ij}\) allows us to replace it using  \eqref{eq:seconddiff-main}.
The remaining two terms in \eqref{eq:Pij-main}, involving
\(v_{13i}\) and \(v_{23i}\), contribute the quadratic expressions
coming from the two positive eigenvalue directions.
After collecting all principal quadratic terms and discarding only
terms controlled by \(C(P+|\nabla P|)\), we have
\begin{equation}\label{eq:Q3-main}
-F^{ij}P_{ij}\sim\mathcal Q(\xi,\eta,\zeta),
\end{equation}
where
\begin{align}
\mathcal Q(\xi,\eta,\zeta)
={}&\lambda_1\lambda_2
\Bigl[
2v^2(1+4v^2v_3^2)(\xi\eta-\zeta^2)
+2v_3F_{Rv}[v_{ij3}]+F_{vv}v_3^2
\Bigr]
\notag\\
&+2\lambda_1\left(F^{11}\zeta^2+2F^{12}\zeta\eta+F^{22}\eta^2\right)
\notag\\
&+2\lambda_2\left(F^{11}\xi^2+2F^{12}\xi\zeta+F^{22}\zeta^2\right).
\label{eq:Q3-expanded-main}
\end{align}

We now eliminate $\lambda_2$ using the original equation. To keep the final polynomial readable, set
\begin{align}
D&=(-v)\lambda_1(1+4v^2v_3^2)-(v_1^2+v_3^2),
\label{eq:Dmain}\\
U&=1+4(-v)\lambda_1(v_2^2+v_3^2),
\label{eq:Umain}\\
N_0&=1+4v_3^2\bigl(v^2+v_1^2+v_2^2+v_3^2\bigr),
\label{eq:N0main}\\
N_1&=N_0+4v_1^2v_2^2.
\label{eq:N1main}
\end{align}
By \eqref{eq:F22-main}, $F^{22}=(-v)D$, hence
\begin{equation}\label{eq:Dpositive-main}
D>0.
\end{equation}
Collecting the $\lambda_2$ terms in \eqref{eq:rank2eq} gives
\begin{equation}\label{eq:lambda2-main}
\lambda_2=\frac{U}{4(-v)D}.
\end{equation}
Substitute \eqref{eq:eta-main} and \eqref{eq:lambda2-main} into \eqref{eq:Q3-expanded-main}. After collecting terms, the reduced two-variable polynomial is
\begin{align}
\mathcal Q_{\mathrm{red}}(\xi,\zeta)
={}&\frac{N_1}{8D^2}\xi^2
+\frac{v_1v_2}{D}\xi\zeta
+\frac12\zeta^2
\notag\\
&-\frac{\lambda_1v_3\bigl(N_1+4v^2v_3^2U\bigr)}{4(-v)D^2}\xi
-\frac{\lambda_1v_1v_2v_3}{(-v)D}\zeta
\notag\\
&+\frac{\lambda_1^2v_3^2}{8v^2D^2}
\Bigl[
1+4(v_1^2+v_3^2)(v_2^2+v_3^2)
+4v^2v_3^2(3U^2+2U+1)
\Bigr].
\label{eq:Qred-main}
\end{align}
Its quadratic coefficient matrix is
\begin{equation}\label{eq:Hmatrix-main}
\begin{pmatrix}
\dfrac{N_1}{8D^2}&\dfrac{v_1v_2}{2D}\\[2mm]
\dfrac{v_1v_2}{2D}&\dfrac12
\end{pmatrix},
\end{equation}
whose determinant is
\begin{equation}\label{eq:Hdet-main}
\frac{N_1}{16D^2}-\frac{v_1^2v_2^2}{4D^2}
=\frac{N_0}{16D^2}>0.
\end{equation}
Thus the quadratic part is positive definite.  In particular, for the fixed lower-order quantities
\(v\), \(Dv\), and \(\lambda_1\), the polynomial
\(Q_{\mathrm{red}}(\xi,\zeta)\) has a unique global minimum.
The value in \eqref{eq:Qmin-main} is obtained by solving the corresponding
two-dimensional linear system for the minimizing pair
\((\xi,\zeta)\) and substituting it back into \eqref{eq:Qred-main}.
Notice that \(D>0\) by \eqref{eq:Dpositive-main} and \(N_0>0\) by its definition, so
the denominator in \eqref{eq:Qmin-main}  is positive. The numerator is manifestly
nonnegative, and hence the minimum is nonnegative. Completing the square gives the exact minimum
\begin{equation}\label{eq:Qmin-main}
\min_{\xi,\zeta}\mathcal{Q}_{\mathrm{red}}(\xi,\zeta)
=
\frac{
\lambda_1^2v_3^4U^2
\left[3+4v_3^2\left(3|Dv|^2+2v^2\right)\right]
}{
2D^2N_0
}
\ge 0.
\end{equation}

Therefore
\[
F^{ij}P_{ij}\le C(P+|\nabla P|).
\]
The strong minimum principle implies that the zero set of $P$ is open in a minimum-rank neighborhood. It is closed as well; connectedness yields constant rank. This proves Theorem~\ref{thm:realn} for n=3.

\medskip
\noindent\textbf{Method II: the inverse-convexity argument in arbitrary
dimension.}
We now give the proof in arbitrary dimension. Instead of relying on the
dimension-three third-derivative reduction, we verify the structural
inverse-matrix convexity condition in the Bian--Guan theorem. The main
algebraic input is the hyperbolic-polynomial inverse-convexity theorem of
Li--Ma--Salani \cite{LMS26}. This argument applies for every $n\ge3$ and, in
particular, also provides an alternative proof in dimension three. Ellipticity follows from Lemma~\ref{lem:real-ellipticity}, while
\begin{equation}\label{eq:F0real}
F(0,p,z)=-\frac14\neq0.
\end{equation}
We verify the level-set condition in Theorem~\ref{thm:BG}.

Fix $p\neq0$ and put
\begin{equation}\label{eq:alpha-real-method2}
\alpha=\frac{p}{|p|},
\qquad
Q=I-\alpha\otimes\alpha.
\end{equation}
Let $A\in\Sym_{++}(n)$ be the inverse-matrix variable and let $z<0$. Set
\begin{equation}\label{eq:Bscale-real}
B=\frac{A}{-z}.
\end{equation}
Then
\[
A^{-1}=\frac1{-z}B^{-1}.
\]
By the homogeneity of $\sigma_2$, $T_1$, and $T_2$, the inequality $F(A^{-1},p,z)\le0$ is equivalent to
\begin{equation}\label{eq:real-level1}
\sigma_2(B^{-1})
-|p|^2\tr(QB^{-1})
+4|p|^2z^2T_2(B^{-1})[\alpha,\alpha]
\le\frac14.
\end{equation}
Since $\tr(QB^{-1})>0$, this is equivalent to
\begin{equation}\label{eq:Hreal-level}
\mathcal H_p(B,z)\le|p|^2,
\end{equation}
where
\begin{align}
\mathcal H_p(B,z)
={}&
\frac{\sigma_2(B^{-1})-\frac14}{\tr(QB^{-1})}
+4|p|^2z^2
\frac{T_2(B^{-1})[\alpha,\alpha]}{\tr(QB^{-1})}.
\label{eq:Hreal}
\end{align}
The first term is convex by Theorem~\ref{thm:LMS} with $\lambda=1/4$.

For the second term define the same hyperbolic polynomial as in the Li--Ma--Salani proof,
\begin{equation}\label{eq:g-real-second}
g(B)=D_Q\det B.
\end{equation}
Then
\begin{equation}\label{eq:g-real-first-id}
g(B)=\det B\,\tr(QB^{-1}).
\end{equation}
Differentiating once more in the direction $Q$ gives
\begin{align}
D_Qg(B)
&=D_Q^2\det B
\notag\\
&=\det B\left[(\tr(QB^{-1}))^2-\tr(QB^{-1}QB^{-1})\right].
\label{eq:DQg-real}
\end{align}
To identify the bracket in \eqref{eq:DQg-real}, choose orthonormal
coordinates so that $\alpha=e_n$. Then
\[
Q=\operatorname{diag}(1,\ldots,1,0).
\]
Writing
\[
B^{-1}=
\begin{pmatrix}
A & b\\
b^T & c
\end{pmatrix},
\]
we have
\[
\operatorname{tr}(QB^{-1})=\operatorname{tr}A,
\qquad
\operatorname{tr}(QB^{-1}QB^{-1})=\operatorname{tr}(A^2).
\]
Moreover, from
$T_2(M)=\sigma_2(M)I-\sigma_1(M)M+M^2$,
\[
T_2(B^{-1})[\alpha,\alpha]
=(T_2(B^{-1}))_{nn}
=\sigma_2(A).
\]
Hence
\[
\bigl(\operatorname{tr}(QB^{-1})\bigr)^2
-\operatorname{tr}(QB^{-1}QB^{-1})
=2T_2(B^{-1})[\alpha,\alpha].
\]
Since the equation holds regardless of the choice of coordinates, the bracket in \eqref{eq:DQg-real} equals $2T_2(B^{-1})[\alpha,\alpha]$. Hence
\begin{equation}\label{eq:Theta-real}
\Theta(B):=
\frac{T_2(B^{-1})[\alpha,\alpha]}{\tr(QB^{-1})}
=\frac12\frac{D_Qg(B)}{g(B)}.
\end{equation}
The polynomial $g=D_Q\det$ is hyperbolic by Theorem~\ref{thm:derivhyper}. By Theorem~\ref{thm:logderivative}, $\Theta$ is convex. Moreover
\begin{equation}\label{eq:Theta-real-hom}
\Theta(tB)=t^{-1}\Theta(B).
\end{equation}
By Lemma~\ref{lem:perspective}, with $a=-z>0$,
\[
(B,a)\longmapsto a^2\Theta(B)
\]
is jointly convex. Therefore $\mathcal H_p(B,-a)$ is jointly convex in $(B,a)$. Since $B^{-1}>0$, one has $T_2(B^{-1})[\alpha,\alpha]\ge0$; hence
\begin{equation}\label{eq:Hreal-monotone}
\frac{\partial}{\partial a}\mathcal H_p(B,-a)
=8|p|^2a\Theta(B)\ge0.
\end{equation}

We now transfer this convexity back to the original variables. Let two points $(A_i,z_i)$ belong to the level set and put
\begin{equation}\label{eq:aiBi-real}
a_i=-z_i>0,
\qquad
B_i=\frac{A_i}{a_i},
\qquad i=1,2.
\end{equation}
For $0\le\theta\le1$, set
\begin{equation}\label{eq:Atheta-real}
A_\theta=\theta A_1+(1-\theta)A_2,
\qquad
a_\theta=\theta a_1+(1-\theta)a_2.
\end{equation}
Then
\begin{equation}\label{eq:weightedB-real}
\frac{A_\theta}{a_\theta}
=\mu_1B_1+\mu_2B_2,
\end{equation}
where
\begin{equation}\label{eq:mureal}
\mu_1=\frac{\theta a_1}{a_\theta},
\qquad
\mu_2=\frac{(1-\theta)a_2}{a_\theta},
\qquad
\mu_1+\mu_2=1.
\end{equation}
Define
\begin{equation}\label{eq:ahat-real}
\widehat a=\mu_1a_1+\mu_2a_2
=\frac{\theta a_1^2+(1-\theta)a_2^2}{a_\theta}.
\end{equation}
A direct calculation gives
\begin{equation}\label{eq:ahat-diff-real}
\widehat a-a_\theta
=\frac{\theta(1-\theta)(a_1-a_2)^2}{a_\theta}\ge0.
\end{equation}
Using monotonicity in $a$, then joint convexity, we obtain
\begin{align}
\mathcal H_p\left(\frac{A_\theta}{a_\theta},-a_\theta\right)
&\le
\mathcal H_p(\mu_1B_1+\mu_2B_2,-\widehat a)
\notag\\
&\le
\mu_1\mathcal H_p(B_1,-a_1)
+\mu_2\mathcal H_p(B_2,-a_2)
\le |p|^2.
\label{eq:return-real}
\end{align}
This proves convexity of the original $(A,z)$ level set for $p\neq0$.

If $p=0$, the level-set condition becomes
\[
z^2\sigma_2(A^{-1})\le\frac14.
\]
After the same scaling $B=A/(-z)$ it becomes
\[
\sigma_2(B^{-1})\le\frac14.
\]
Fixing any unit vector $\alpha$ and $Q=I-\alpha\otimes\alpha$, this is the zero sublevel set of the convex function
\[
B\longmapsto\frac{\sigma_2(B^{-1})-\frac14}{\tr(QB^{-1})}.
\]
Thus the $B$-set is convex, and \eqref{eq:weightedB-real} transfers convexity back to $(A,z)$.

All hypotheses of Theorem~\ref{thm:BG} are verified, and Theorem~\ref{thm:realn} follows.
\begin{remark}\label{rem:mean-curvature}
The same square-root and inverse-convexity mechanism explains the exponent in the mean-curvature equation. With the non-normalized convention
\[
H[u]=\operatorname{div}\left(\frac{Du}{\sqrt{1+|Du|^2}}\right),
\]
one has
\[
H[u]
=\frac{(1+|Du|^2)\Delta u-u_iu_ju_{ij}}{(1+|Du|^2)^{3/2}}.
\]
Consider $H[u]=\omega^b$ and set $u=-v^2$. The constant rank theorem has been proved for $b=-3$ in \cite{KL87}.
Since
\[
Du=-2vDv,
\qquad
D^2u=-2vD^2v-2Dv\otimes Dv,
\]
a direct substitution gives
\begin{align}
&(-v)\left[\Delta v+4v^2\bigl(|Dv|^2\Delta v-v_iv_jv_{ij}\bigr)\right]-|Dv|^2
\notag\\
&\hspace{35mm}
=\frac12\bigl(1+4v^2|Dv|^2\bigr)^{(b+3)/2}.
\label{eq:mean-general}
\end{align}
Hence the right-hand side is constant if and only if $b=-3$. For this exponent the equation is
\begin{equation}\label{eq:mean-v}
F_H(D^2v,Dv,v)=0,
\end{equation}
with
\begin{align}
F_H(R,p,z)
={}&(-z)\tr\left(\left[I+4z^2(|p|^2I-p\otimes p)\right]R\right)
-|p|^2-\frac12.
\label{eq:mean-F}
\end{align}
The Bian--Guan level-set condition is especially simple. Put $a=-z>0$ and $B=A/a$. Then
\begin{align}\label{eq:mean-level}
F_H(A^{-1},p,-a)\le0
\quad\Longleftrightarrow\quad
&\tr(B^{-1})\notag
+4a^2\tr\bigl((|p|^2I-p\otimes p)B^{-1}\bigr)
\le |p|^2+\frac12.
\end{align}
Put $P=|p|^2I-p\otimes p$. Both functions
\[
B\mapsto\tr(B^{-1})=\frac{D_I\det B}{\det B}
\]
and
\[
B\mapsto\tr(PB^{-1})=\frac{D_P\det B}{\det B}
\]
are convex by the hyperbolic logarithmic-derivative principle. The second one is homogeneous of degree $-1$, so $a^2\tr(PB^{-1})$ is jointly convex by the perspective lemma. It is also nondecreasing in $a$. The same reweighted argument \eqref{eq:weightedB-real}--\eqref{eq:ahat-diff-real} therefore proves the original $(A,z)$ level-set convexity. Consequently, if $D^2v\ge0$, the Bian--Guan theorem gives constant rank for the transformed mean-curvature equation with exponent $-3$. If one uses normalized mean curvature $H/n$, only the constant $1/2$ changes; the exponent $-3$ does not.
\end{remark}

The constant-rank theorem obtained above is a local statement: once
$D^2v\ge 0$ is known, degeneration of its rank cannot occur at an
isolated interior point. We now use this result as the closedness
mechanism in a global deformation argument. The remaining ingredients
are strict convexity near the boundary, a strictly convex solution on
the unit ball, solvability and uniform estimates along a deformation
of the domain, and local stability of the corresponding solutions.

\subsection{Boundary strict convexity}
\label{subsubsec:real-boundary}

We use the following standard boundary convexity lemma due to
Korevaar, see \cite{Korevaar1983}. 

\begin{lemma}
\label{lem:boundary-convexity-general}
Let $\Omega\subset\mathbb R^N$ be bounded, smooth, and uniformly
strictly convex, and let $u\in C^2(\overline{\Omega})$ satisfy
$u<0$ in $\Omega$, $u=0$ on $\partial\Omega$, and
$Du\cdot\mathbf n>0$ on $\partial\Omega$, where $\mathbf n$ is the
exterior unit normal.

Let $f\in C^2((-\infty,0))$ satisfy
\[
f'>0,\qquad f''>0,\qquad
\lim_{t\to0^-}\frac{f'(t)}{f''(t)}=0.
\]
Then there exists $\delta>0$ such that $f(u)$ is strictly convex in
the boundary strip
\[
\{x\in\Omega:0<\operatorname{dist}(x,\partial\Omega)<\delta\}.
\]
\end{lemma}

We now apply Lemma~\ref{lem:boundary-convexity-general} to the
transformation used in this paper. Let
\[
f(t)=-\sqrt{-t},\qquad t<0.
\]
Then
\[
f'(t)=\frac{1}{2\sqrt{-t}}>0,
\qquad
f''(t)=\frac{1}{4(-t)^{3/2}}>0,
\]
and
\[
\frac{f'(t)}{f''(t)}
=2(-t)\longrightarrow0
\qquad\text{as }t\to0^-.
\]
Thus it remains only to verify the Hopf condition for the solution.

\begin{lemma}
\label{lem:real-boundary-strict}
Let $u$ be an admissible solution of \eqref{eq:real-main}, and set
$v=-\sqrt{-u}$. Then there exists $\delta>0$ such that
$D^2v>0$ whenever $0<\operatorname{dist}(x,\partial\Omega)<\delta$.
\end{lemma}

\begin{proof}
Since the graph is $\Gamma_2$-admissible, the equation is elliptic
on the admissible branch. Together with the zero Dirichlet condition
and the strong maximum principle, this gives $u<0$ in $\Omega$.
The Hopf boundary point lemma then yields
$\partial_{\mathbf n}u>0$ on $\partial\Omega$.
The assertion now follows directly from
Lemma~\ref{lem:boundary-convexity-general} with
$f(t)=-\sqrt{-t}$.
\end{proof}

\subsection{The unit ball}
\label{subsubsec:real-ball}

The boundary lemma above provides strict convexity near
$\partial\Omega$ for any sufficiently regular admissible solution,
but it does not provide a globally strictly convex solution from
which the deformation can start. We now construct such a reference
solution on the unit ball. Before doing so, we rewrite
\eqref{eq:real-main} in a form with constant right-hand side; this
form will also be used in the continuity argument of the next
subsection.

To obtain this formulation, we first isolate the algebraic identity already
obtained in \eqref{eq:real-metric-sigma2}. Replacing $2vDv$ there by a general vector $\vartheta$, we are naturally led to the following
notation.

For $M\in\mathrm{Sym}(n)$ and $\vartheta\in\mathbb R^n$, define
\begin{equation}\label{eq:E-definition}
\mathcal E(M,\vartheta)
:=
\sigma_2(M)+T_2(M)[\vartheta,\vartheta].
\end{equation}
Set
\[
\widehat M_\vartheta
=
(I+\vartheta\otimes\vartheta)^{-1/2}
M
(I+\vartheta\otimes\vartheta)^{-1/2}.
\]
Since
$(I+\vartheta\otimes\vartheta)^{-1}M$
is similar to $\widehat M_\vartheta$, the two matrices have the same
elementary symmetric functions.
By the same calculation as in \eqref{eq:real-metric-sigma2}, we have
\[
(1+|\vartheta|^2)\sigma_2(\widehat M_\vartheta)
=
\mathcal E(M,\vartheta).
\]

For a graph, the symmetric representative of the Weingarten map is
\[
\frac{1}{\sqrt{1+|Du|^2}}
(I+Du\otimes Du)^{-1/2}
D^2u
(I+Du\otimes Du)^{-1/2}.
\]
It follows that
\[
(1+|Du|^2)^2\sigma_2(\kappa[u])
=
\mathcal E(D^2u,Du).
\]
Consequently, equation \eqref{eq:real-main} is exactly equivalent
to
\begin{equation}\label{eq:real-main-1}
\mathcal E(D^2u,Du)=1.
\end{equation}

We also record the linearization of \eqref{eq:real-main-1}, which will
be used repeatedly in the deformation argument. For
$N\in\mathrm{Sym}(n)$,
\begin{equation}\label{eq:E-Hessian-linearization}
D_M\mathcal E(M,\vartheta)[N]
=
(1+|\vartheta|^2)
T_1(\widehat M_\vartheta):
\Bigl[
(I+\vartheta\otimes\vartheta)^{-1/2}
N
(I+\vartheta\otimes\vartheta)^{-1/2}
\Bigr].
\end{equation}
If the graph is $\Gamma_2$-admissible, then
$\widehat M_\vartheta\in\Gamma_2$, and hence
$T_1(\widehat M_\vartheta)>0$. Thus
\eqref{eq:real-main-1} is elliptic on the admissible branch.

Moreover, for $\chi\in\mathbb R^n$,
\[
D_\vartheta\mathcal E(M,\vartheta)[\chi]
=
2T_2(M)[\vartheta,\chi].
\]
Therefore the linearized operator at a smooth admissible solution has
the form
\begin{equation}\label{eq:real-linearized}
\mathcal L\phi=a^{ij}\phi_{ij}+d^k\phi_k,
\end{equation}
with $(a^{ij})>0$ and no zero-order term.

We now turn to the unit ball. The following lemma gives the initial
point of the domain deformation and does not use the constant-rank
theorem.

\begin{lemma}
\label{lem:real-ball}
On the unit ball $B_1$, the Dirichlet problem
\[
\mathcal E(D^2u,Du)=1
\quad\text{in }B_1,
\qquad
u=0
\quad\text{on }\partial B_1
\]
admits a smooth radial $\Gamma_2$-admissible solution $u_0$.
If
$v_0=-\sqrt{-u_0},$
then
$D^2v_0>0\,\,\text{in }B_1.$
\end{lemma}

\begin{proof}
We seek a radial solution $u_0=u_0(r)$, where $r=|x|$, and set
$\ell(r)=u_0'(r)^2$. The eigenvalues of $D^2u_0$ are $u_0''(r)$ and
$u_0'(r)/r$ with multiplicity $n-1$. Since
$Du_0=u_0'(r)e_r$, equation \eqref{eq:real-main-1} becomes
\[
(n-1)\frac{u_0'u_0''}{r}
+
\binom{n-1}{2}
\frac{u_0'^2}{r^2}(1+u_0'^2)
=
1.
\]
Equivalently,
\begin{equation}\label{eq:ball-ode}
r\ell'+(n-2)\ell(1+\ell)
=
\frac{2r^2}{n-1},
\qquad
\ell(0)=0.
\end{equation}

The smoothness condition at the center gives
\begin{equation}\label{eq:ball-center}
\ell(r)
=
\frac{2}{n(n-1)}r^2+O(r^4).
\end{equation}
Hence $\ell>0$ and $\ell'>0$ for all sufficiently small $r>0$.

We claim that in fact $\ell'(r)>0$ for every $0<r\le1$. Suppose
otherwise, and let $r_0>0$ be the first zero of $\ell'$. Differentiating
\eqref{eq:ball-ode} gives
\[
\ell'+r\ell''+(n-2)(1+2\ell)\ell'
=
\frac{4r}{n-1}.
\]
Evaluating at $r_0$ yields
\[
r_0\ell''(r_0)
=
\frac{4r_0}{n-1}>0,
\]
whereas the first-zero property of $\ell'$ requires
$\ell''(r_0)\le0$. This contradiction proves
\begin{equation}\label{eq:ell-positive}
\ell'(r)>0
\qquad\text{for }0<r\le1.
\end{equation}

Equation \eqref{eq:ball-ode} also gives
$\ell'(r)\le2r/(n-1)$ whenever $\ell\ge0$. Integrating from $0$ to $r$ and using
$\ell(0)=0$, we obtain
\[
0\le \ell(r)\le \frac{r^2}{n-1},
\]
so no finite-radius blow-up can occur. We therefore take the positive branch
$u_0'=\sqrt{\ell}$ and impose the zero boundary condition by setting
\begin{equation}\label{eq:ball-solution}
u_0(r)
=
-\int_r^1\sqrt{\ell(t)}\,dt.
\end{equation}
Then $u_0<0$ in $B_1$ and $u_0=0$ on $\partial B_1$. Moreover,
$u_0'>0$, $u_0''=\ell'/(2\sqrt{\ell})>0$, and $u_0'/r>0$.
Consequently,
\[
D^2u_0>0
\qquad\text{in }B_1.
\]
In particular, the graph of $u_0$ is $\Gamma_2$-admissible.

Finally, since
\[
D^2v_0
=
\frac{1}{2\sqrt{-u_0}}D^2u_0
+
\frac{1}{4(-u_0)^{3/2}}
Du_0\otimes Du_0,
\]
the first term is positive definite and the second is positive
semidefinite. Hence
\[
D^2v_0>0
\qquad\text{in }B_1.
\]
\end{proof}

Lemma~\ref{lem:real-ball} supplies a full-rank starting point for the
strict-convexity argument. We next deform the unit ball to the
prescribed domain and show that the corresponding admissible
Dirichlet solutions exist with estimates that are uniform along the
entire deformation path.
\subsection{Minkowski deformation, solvability, and stability}
\label{subsubsec:real-deformation}

We now deform the unit ball to the prescribed domain. The main point
of this subsection is to show that the admissible Dirichlet problem
remains solvable along the whole deformation path and that the
corresponding solutions satisfy estimates uniform in the deformation
parameter. These estimates will also give the local stability needed
in the proof of Theorem~\ref{thm:real-strict-convexity}.

After a translation, we may assume that $0\in\Omega$. For
$0\le s\le1$, define
\begin{equation}\label{eq:minkowski-path}
\Omega_s=(1-s)B_1+s\Omega.
\end{equation}
Thus $\Omega_0=B_1$ and $\Omega_1=\Omega$.

Let $\mathfrak{s}_K$ denote the support function of a convex body $K$.
Then
\[
\mathfrak{s}_{\Omega_s}
=(1-s)+s\mathfrak{s}_{\Omega}.
\]
The curvature-radius matrix of $\partial\Omega_s$ is
\[
\nabla_{\mathbb S^{n-1}}^2\mathfrak{s}_{\Omega_s}
+\mathfrak{s}_{\Omega_s}g_{\mathbb S^{n-1}},
\]
which is a convex combination of the corresponding positive definite
matrices for $B_1$ and $\Omega$. Hence $\{\Omega_s\}_{0\le s\le1}$
is a smooth uniformly strictly convex family. In particular, the
diameters, boundary curvatures, and the boundary geometry of any fixed
finite order are uniformly controlled.

The Dirichlet problem for prescribed curvature equations on graphs
has been extensively studied; see, among others,
Caffarelli--Nirenberg--Spruck \cite{CNS88},
Ivochkina \cite{Ivo90,Ivo91}, and Trudinger \cite{Tru90}.
These works provide classical existence and a priori estimate theories
for broad classes of prescribed-curvature Dirichlet problems.
They do not, however, apply directly to the present equation, since here
the prescribed curvature depends explicitly on the unit normal:
\[
\sigma_2(\kappa[u])
=(1+|Du|^2)^{-2}
=\langle \nu,e_{n+1}\rangle^4.
\]
In particular, the right-hand side is not of the purely position-dependent
form treated in the classical Weingarten Dirichlet theorem of
Caffarelli--Nirenberg--Spruck, while the hypotheses in the corresponding
existence results of Ivochkina and Trudinger do not directly yield the
smooth admissible solvability required here.

We therefore give a continuity argument for the present equation.
The main external ingredient is the global curvature estimate of
Guan--Ren--Wang \cite{GRW15}, which applies to admissible
$2$-convex solutions of equations of the form
\[
\sigma_2(\kappa)=f(X,\nu)
\]
and, in the graphical Dirichlet setting, controls the interior second
derivatives in terms of the boundary $C^2$ norm.
Together with the boundary estimates established below, this provides
the uniform $C^2$ bound needed to close the continuity method.

\begin{proposition}
\label{prop:real-solvability}
For every $s\in[0,1]$, the Dirichlet problem
\begin{equation}\label{eq:deformation-problem}
\begin{cases}
\mathcal E(D^2u_s,Du_s)=1,
& x\in\Omega_s,\\
u_s=0,
& x\in\partial\Omega_s.
\end{cases}
\end{equation}
admits a unique smooth $\Gamma_2$-admissible solution $u_s$.
Moreover, there exist $\alpha_0\in(0,1)$ and a constant $C>0$,
independent of $s\in[0,1]$, such that
\begin{equation}\label{eq:uniform-c2a}
\|u_s\|_{C^{2,\alpha_0}(\overline{\Omega_s})}\le C.
\end{equation}
\end{proposition}

\begin{proof}

All estimates below are a priori estimates. Fix $s\in[0,1]$
and let $u_s$ be the smooth $\Gamma_2$-admissible solution of
\eqref{eq:deformation-problem}. We prove estimates whose constants are independent of $s$.

\medskip
\noindent\textbf{Step 1. Uniform $C^0$ estimate and boundary normal derivative.}

Since $\{\Omega_s\}$ is a smooth uniformly strictly convex family,
we may choose smooth defining functions $\beta_s$ such that
$\beta_s<0$ in $\Omega_s$, $\beta_s=0$ on $\partial\Omega_s$, and
\[
D^2\beta_s\ge c_\beta I
\]
for some $c_\beta>0$ independent of $s$. The relevant finite-order
norms of $\beta_s$, as well as positive upper and lower bounds for
$\partial_{\mathbf n_s}\beta_s$, are also uniform in $s$, where
$\mathbf n_s$ denotes the exterior unit normal to $\partial\Omega_s$.

Choose constants $C_+\gg1$ and $0<c_-\ll1$, independent of $s$.
By the homogeneity of $\sigma_2$ and $T_2$,
\[
\mathcal E(D^2(C_+\beta_s),D(C_+\beta_s))
=
C_+^2\sigma_2(D^2\beta_s)
+
C_+^4T_2(D^2\beta_s)[D\beta_s,D\beta_s],
\]
and
\[
\mathcal E(D^2(c_-\beta_s),D(c_-\beta_s))
=
c_-^2\sigma_2(D^2\beta_s)
+
c_-^4T_2(D^2\beta_s)[D\beta_s,D\beta_s].
\]
Since $D^2\beta_s>0$, $T_2(D^2\beta_s)\ge0$. Thus $C_+$ and $c_-$
can be chosen uniformly so that
\[
\mathcal E(D^2(C_+\beta_s),D(C_+\beta_s))>1,
\qquad
\mathcal E(D^2(c_-\beta_s),D(c_-\beta_s))<1.
\]
Moreover, since $D^2\beta_s>0$, both $C_+\beta_s$ and
$c_-\beta_s$ lie on the $\Gamma_2$-admissible branch. They may
therefore be used as admissible comparison functions.
The comparison principle gives
\begin{equation}\label{eq:uniform-C0-barriers}
C_+\beta_s\le u_s\le c_-\beta_s<0.
\end{equation}
In particular, $\|u_s\|_{L^\infty(\Omega_s)}\le C$.

Since all three functions vanish on $\partial\Omega_s$, the Hopf boundary lemma applied to the two comparison functions gives
\[
c_-\partial_{\mathbf n_s}\beta_s
\le \partial_{\mathbf n_s}u_s
\le C_+\partial_{\mathbf n_s}\beta_s.
\]
Consequently,
\begin{equation}\label{eq:uniform-hopf}
0<c_0\le\partial_{\mathbf n_s}u_s\le C_0
\qquad\text{on }\partial\Omega_s,
\end{equation}
where $c_0$ and $C_0$ are independent of $s$.

\medskip
\noindent\textbf{Step 2. Uniform $C^1$ estimate.}

Let $\mathcal L_s$ be the linearization of
\eqref{eq:E-definition} at $u_s$. Since the equation is independent of
$x$ and $u$, differentiation with respect to $x_k$ gives
$\mathcal L_s((u_s)_k)=0$. Therefore
\[
\mathcal L_s\left(\frac12|Du_s|^2\right)
=
a_s^{ij}(u_s)_{ki}(u_s)_{kj}\ge0.
\]
The maximum principle yields
\[
\sup_{\Omega_s}|Du_s|
=
\sup_{\partial\Omega_s}|Du_s|.
\]
Since $u_s=0$ on $\partial\Omega_s$, the boundary gradient is purely
normal. Hence \eqref{eq:uniform-hopf} gives
\begin{equation}\label{eq:uniform-c1}
\|Du_s\|_{L^\infty(\Omega_s)}\le C.
\end{equation}
In particular, if $\nu_s$ denotes the upward unit normal to the graph,
then
\begin{equation}\label{eq:uniform-angle}
\langle\nu_s,e_{n+1}\rangle
=
\frac{1}{\sqrt{1+|Du_s|^2}}
\ge c_1>0.
\end{equation}

\medskip
\noindent\textbf{Step 3. Boundary $C^2$ estimate.}

Fix a point of $\partial\Omega_s$ and choose an orthonormal frame such
that $e_n=\mathbf n_s$ and $e_1,\ldots,e_{n-1}$ are principal
directions of $\partial\Omega_s$. Differentiating the zero boundary
condition twice in tangential directions gives
\[
(u_s)_{\alpha\beta}
=
\partial_{\mathbf n_s}u_s\,
(\mathrm{II}_{\partial\Omega_s})_{\alpha\beta}.
\]
Uniform strict convexity and \eqref{eq:uniform-hopf} imply
\begin{equation}\label{eq:boundary-tangential}
c_2I\le
\bigl((u_s)_{\alpha\beta}\bigr)_{\alpha,\beta<n}
\le C_2I.
\end{equation}

We next estimate the mixed derivatives. Near the chosen boundary
point, with the exterior-normal convention used above, write
\[
x_n
=
-\frac12\sum_{\alpha<n}
k_\alpha^{(s)}x_\alpha^2+O(|x'|^3),
\]
where $k_\alpha^{(s)}>0$ are the principal curvatures of
$\partial\Omega_s$. Define
\[
Z_\alpha
=
\partial_\alpha
-
k_\alpha^{(s)}
\bigl(x_\alpha\partial_n-x_n\partial_\alpha\bigr).
\]
The vector field $Z_\alpha$ is a constant linear combination of
an infinitesimal translation and an infinitesimal Euclidean
rotation. Since \eqref{eq:deformation-problem} is invariant under both translations
and rotations, differentiating these invariances at the identity
gives
\[
L_s(Z_\alpha u_s)=0.
\]
The constants entering this identity are uniformly controlled
because the principal curvatures of $\partial\Omega_s$ and the
relevant boundary coordinate charts are uniform in $s$,
while the boundary expansion gives
$|Z_\alpha u_s|\le C|x|^2$ on the local boundary.

Using the uniformly strict comparison functions from Step~1 in the
standard Caffarelli--Nirenberg--Spruck boundary barrier argument
\cite{CNS88}, we obtain
\begin{equation}\label{eq:boundary-mixed}
|(u_s)_{\alpha n}|\le C_3
\qquad\text{on }\partial\Omega_s.
\end{equation}

It remains to estimate $(u_s)_{nn}$. Set
\[
J_s=
\bigl((u_s)_{\alpha\beta}\bigr)_{\alpha,\beta<n}.
\]
Since $Du_s=(\partial_{\mathbf n_s}u_s)e_n$ on $\partial\Omega_s$,
the block identities give
\[
\sigma_2(D^2u_s)
=
\sigma_2(J_s)
+
(u_s)_{nn}\operatorname{tr}J_s
-
\sum_{\alpha<n}(u_s)_{\alpha n}^2
\]
and
\[
T_2(D^2u_s)_{nn}=\sigma_2(J_s).
\]
Substitution into \eqref{eq:deformation-problem} yields
\[
(u_s)_{nn}
=
\frac{
1+\displaystyle\sum_{\alpha<n}(u_s)_{\alpha n}^2
-
\bigl(1+(\partial_{\mathbf n_s}u_s)^2\bigr)\sigma_2(J_s)
}{
\operatorname{tr}J_s
}.
\]
Using \eqref{eq:uniform-hopf},
\eqref{eq:boundary-tangential}, and
\eqref{eq:boundary-mixed}, we obtain
\begin{equation}\label{eq:boundary-c2}
\|D^2u_s\|_{L^\infty(\partial\Omega_s)}\le C.
\end{equation}

\medskip
\noindent\textbf{Step 4. Global $C^2$ estimate.}

Returning to the geometric form \eqref{eq:real-main},
\[
\sigma_2(\kappa[u_s])
=
\langle\nu_s,e_{n+1}\rangle^4.
\]
By \eqref{eq:uniform-angle}, the right-hand side is bounded away from
zero uniformly in $s$.

To apply the prescribed-curvature estimate, we verify a uniform
star-shapedness condition. Write $Y=(x,u_s(x))$ and choose a fixed
point $Y_0=(0,-L_0)$ with $L_0$ sufficiently large. The uniform
$C^0$ and $C^1$ estimates imply
\[
\langle Y-Y_0,\nu_s\rangle
=
\frac{L_0+u_s-x\cdot Du_s}
{\sqrt{1+|Du_s|^2}}
\ge c_3>0
\]
uniformly in $s$. Thus all graphs are uniformly star-shaped with
respect to $Y_0$.  After translating $Y_0$ to the origin and writing the hypersurfaces
in radial form $Y-Y_0=\rho_s(\theta)\theta$, the uniform $C^0$ and
$C^1$ estimates above, together with the uniform lower bound for the
support function, imply
$\inf \rho_s\ge c>0$ and $\|\rho_s\|_{C^1}\le C$ uniformly in $s$. This verifies the star-shapedness
assumption required in the global curvature estimate of
Guan--Ren--Wang \cite{GRW15}.

Here the prescribed function is
\[
f(Y,\nu)=\langle \nu,e_{n+1}\rangle^4.
\]
It is smooth and independent of $Y$, and
\eqref{eq:uniform-angle} gives a uniform positive lower bound for
$f$ along the family. By \eqref{eq:uniform-angle}, we also know the unit normal bundles of the whole family remain in a
fixed spherical cap. Hence they admit a fixed open neighborhood
$\Gamma$ on which
$f(Y,\nu)=\langle\nu,e_{n+1}\rangle^4$ is positive and satisfies
$\inf_{\Gamma}f\ge c>0$ and $\|f\|_{C^2(\Gamma)}\le C$ uniformly in
$s$. Thus the assumptions on the prescribed function in
\cite{GRW15} are satisfied. Moreover,
$\kappa[u_s]\in\Gamma_2$ is precisely the $2$-convexity
assumption. Thus the structural quantities entering the
$k=2$ prescribed-curvature estimate are uniformly controlled
along the deformation.

The global curvature estimate for admissible $2$-convex
prescribed-curvature equations, together with its boundary version
and \eqref{eq:boundary-c2}, gives
\[
\max_{\Sigma_{u_s}}\max_i\kappa_i[u_s]\le C;
\]
see \cite{GRW15}. Since $\kappa[u_s]\in\Gamma_2$, we also have
$\sigma_1(\kappa[u_s])>0$. If
$\kappa_1\ge\cdots\ge\kappa_n$, then
\[
\kappa_n
>
-\sum_{i=1}^{n-1}\kappa_i
\ge -(n-1)C.
\]
Hence all principal curvatures are uniformly bounded from both sides.

Using the graphical expression for the second fundamental form
together with \eqref{eq:uniform-c1}, we conclude that
\begin{equation}\label{eq:global-c2}
\|D^2u_s\|_{L^\infty(\Omega_s)}\le C.
\end{equation}

\medskip
\noindent\textbf{Step 5. $C^{2,\alpha}$ estimate and the continuity method.}

By \eqref{eq:global-c2} and \eqref{eq:uniform-c1}, the principal
curvatures are uniformly bounded. Moreover, by
\eqref{eq:uniform-angle} and the equation,
\[
\sigma_2(\kappa[u_s])
=
\langle\nu_s,e_{n+1}\rangle^4
\ge c_1^4>0.
\]
Since $\kappa[u_s]\in\Gamma_2$, we have
$\sigma_1(\kappa[u_s])>0$, and the identity
\[
\sigma_1(\kappa)^2
=
|\kappa|^2+2\sigma_2(\kappa)
\]
gives
\[
\sigma_1(\kappa[u_s])\ge \sqrt{2}\,c_1^2.
\]
Thus the curvature vectors remain in a compact subset of
$\Gamma_2$, and hence the equation is uniformly elliptic along
the family.

For fixed $\vartheta$, the identity obtained in
Subsection~\ref{subsubsec:real-ball} shows that
$M\mapsto\mathcal E(M,\vartheta)^{1/2}$ is, up to a positive
factor, the composition of $\sigma_2^{1/2}$ with the linear
congruence map
\[
M\longmapsto
(I+\vartheta\otimes\vartheta)^{-1/2}
M
(I+\vartheta\otimes\vartheta)^{-1/2}.
\]
It is therefore concave in $M$ on the admissible branch.
Since the dependence on $\vartheta=Du_s$ is smooth and
\eqref{eq:uniform-c1} gives a uniform bound for $Du_s$,
Evans--Krylov theory and the boundary Schauder estimates yield
\eqref{eq:uniform-c2a}.

Define
\[
\mathcal S_{\mathrm{solv}}
=
\left\{
s\in[0,1]:
\eqref{eq:deformation-problem}
\text{ admits a smooth admissible solution}
\right\}.
\]
Lemma~\ref{lem:real-ball} gives $0\in\mathcal S_{\mathrm{solv}}$.

Let $s_0\in\mathcal S_{\mathrm{solv}}$. For $s$ close to $s_0$,
choose smooth diffeomorphisms
$\Upsilon_s:\Omega_{s_0}\to\Omega_s$ satisfying
$\Upsilon_{s_0}=\operatorname{id}$. After pulling the equation back
to the fixed domain $\Omega_{s_0}$, the linearized Dirichlet operator
at $s=s_0$ has the form \eqref{eq:real-linearized}. It is uniformly
elliptic and has no zero-order term, so the maximum principle gives
a trivial Dirichlet kernel. Schauder theory and the Fredholm
alternative imply invertibility, and the implicit function theorem
shows that $\mathcal S_{\mathrm{solv}}$ is open.

Now let $s_j\in\mathcal S_{\mathrm{solv}}$ and $s_j\to s_\infty$.
After pulling the solutions back to a fixed domain,
\eqref{eq:uniform-c2a} yields a subsequence converging in
$C^{2,\beta}$ for every $0<\beta<\alpha_0$. The limit solves the
corresponding Dirichlet problem on $\Omega_{s_\infty}$.

The limiting curvature vector belongs to
$\overline{\Gamma}_2$. Moreover, by \eqref{eq:uniform-angle},
\[
\sigma_2(\kappa)
=
\langle\nu,e_{n+1}\rangle^4
\ge c_1^4>0.
\]
Since $\sigma_1(\kappa)\ge0$ on $\overline{\Gamma}_2$ and
\[
\sigma_1(\kappa)^2=|\kappa|^2+2\sigma_2(\kappa),
\]
we have $\sigma_1(\kappa)>0$. Thus the limiting curvature vector
still lies in $\Gamma_2$. In particular, the limiting equation is
uniformly elliptic. Standard Schauder bootstrapping then upgrades
the $C^{2,\beta}$ limit to a smooth admissible solution. Hence
$s_\infty\in\mathcal S_{\mathrm{solv}}$, so
$\mathcal S_{\mathrm{solv}}$ is closed. Therefore
\[
\mathcal S_{\mathrm{solv}}=[0,1].
\]

Finally, uniqueness follows from the strong tangency principle for
the admissible Weingarten equation. Indeed, let $u$ and $\widetilde u$
be two admissible solutions with the same zero boundary data. If
\[
c=\max_{\overline{\Omega_s}}(u-\widetilde u)>0,
\]
then the maximum is attained at an interior point, and
$\widetilde u+c$ lies above $u$ and touches it there. Since the
equation is invariant under vertical translations,
$\widetilde u+c$ satisfies the same equation. The strong tangency
principle therefore gives
\[
u\equiv \widetilde u+c,
\]
which contradicts the common zero boundary data. Interchanging
$u$ and $\widetilde u$ gives the reverse inequality, and hence
$u\equiv\widetilde u$.
\end{proof}

The continuity argument also gives the stability of the solutions
with respect to the deformation parameter, which will be used in the
open--closed argument for strict convexity.

\begin{lemma}
\label{lem:real-local-stability}
Fix $s_0\in[0,1]$. For $s$ sufficiently close to $s_0$, let
$\Upsilon_s:\Omega_{s_0}\to\Omega_s$ be the diffeomorphisms used
above. Then
\[
u_s\circ\Upsilon_s
\longrightarrow u_{s_0}
\quad\text{in }
C^{2,\alpha_0}(\overline{\Omega_{s_0}})
\]
as $s\to s_0$.

If $\Omega'\Subset\Omega_{s_0}$ and
$v_s=-\sqrt{-u_s}$, then
\[
D^2v_s(\Upsilon_s(x))
\longrightarrow D^2v_{s_0}(x)
\]
uniformly for $x\in\Omega'$.
\end{lemma}

\begin{proof}
After the equations on $\Omega_s$ are pulled back to
$\Omega_{s_0}$, the implicit function theorem used above, together
with uniqueness, gives
\[
u_s\circ\Upsilon_s\to u_{s_0}
\quad\text{in }C^{2,\alpha_0}(\overline{\Omega_{s_0}}).
\]

Set $\widetilde u_s=u_s\circ\Upsilon_s$. The chain rule gives
\[
D^2\widetilde u_s
=
(D\Upsilon_s)^T
(D^2u_s\circ\Upsilon_s)
D\Upsilon_s
+
\sum_{k=1}^n
\bigl((u_s)_k\circ\Upsilon_s\bigr)
D^2(\Upsilon_s)^k.
\]
Since $\Upsilon_s\to\operatorname{id}$ in $C^2$, it follows that
\[
D^2u_s(\Upsilon_s(x))
\longrightarrow D^2u_{s_0}(x)
\]
uniformly on $\overline{\Omega_{s_0}}$.

Now let $\Omega'\Subset\Omega_{s_0}$. Since
$u_{s_0}<0$ in the interior, $u_{s_0}\le-c_{\Omega'}<0$ on
$\Omega'$. The map $t\mapsto-\sqrt{-t}$ is smooth on the relevant
range, and therefore
$v_s\circ\Upsilon_s\to v_{s_0}$ in $C^2(\Omega')$. Applying the same
chain-rule argument gives the claimed convergence of the physical
Hessians.
\end{proof}

We have thus connected the unit-ball solution to the solution on
$\Omega$ through a family of admissible solutions with uniform
a priori estimates and local $C^2$ stability. In the next subsection
we combine this deformation with the boundary convexity lemma and
the real constant-rank theorem to prove
Theorem~\ref{thm:real-strict-convexity}.
\subsection{Proof of Theorem~\ref{thm:real-strict-convexity}}
\label{subsubsec:proof-real-strict}

We now combine the boundary convexity, the unit-ball model, and the
deformation results established above. The only remaining point is
closedness: a limit of strictly convex solutions is a priori only
weakly convex. This is precisely where the constant-rank theorem of
Section~3.1 enters.

Let $u_s$ be the unique admissible solution on $\Omega_s$ given by
Proposition~\ref{prop:real-solvability}, and set
$v_s=-\sqrt{-u_s}$.

We first observe that the boundary strip in
Lemma~\ref{lem:real-boundary-strict} can be chosen uniformly in $s$.
Indeed, the proof of the boundary convexity lemma depends only on a
positive lower bound for the second fundamental form of
$\partial\Omega_s$, a positive lower bound for
$\partial_{\mathbf n_s}u_s$, a uniform $C^2$ bound for $u_s$, and
uniform control of the boundary geometry. These quantities are
uniform along the Minkowski path by
Proposition~\ref{prop:real-solvability}. Hence there exists
$\delta_0>0$, independent of $s$, such that
\begin{equation}\label{eq:uniform-boundary-strip}
D^2v_s>0
\qquad\text{whenever}\qquad
0<\operatorname{dist}(x,\partial\Omega_s)<\delta_0.
\end{equation}

Define
\[
\mathcal I_{\mathrm{cvx}}
=
\{s\in[0,1]:D^2v_s>0\text{ in }\Omega_s\}.
\]

By Lemma~\ref{lem:real-ball}, $D^2v_0>0$ in $B_1$. Therefore
$0\in\mathcal I_{\mathrm{cvx}}$, so
$\mathcal I_{\mathrm{cvx}}$ is nonempty.

We next prove that $\mathcal I_{\mathrm{cvx}}$ is open. Let
$s_0\in\mathcal I_{\mathrm{cvx}}$ and consider the compact set
\[
K_{s_0}
=
\left\{
x\in\Omega_{s_0}:
\operatorname{dist}(x,\partial\Omega_{s_0})
\ge\frac{\delta_0}{2}
\right\}.
\]
Since $D^2v_{s_0}>0$ in $\Omega_{s_0}$ and $K_{s_0}$ is compact,
there exists $c_{\mathrm{int}}>0$ such that
\[
D^2v_{s_0}\ge 2c_{\mathrm{int}}I
\qquad\text{on }K_{s_0}.
\]
By Lemma~\ref{lem:real-local-stability}, for $s$ sufficiently close
to $s_0$,
\[
D^2v_s(\Upsilon_s(x))
\ge c_{\mathrm{int}}I
\qquad\text{for }x\in K_{s_0}.
\]
Since $\Upsilon_s\to\operatorname{id}$ smoothly and maps
$\partial\Omega_{s_0}$ onto $\partial\Omega_s$, for $s$ sufficiently
close to $s_0$ the part of $\Omega_s$ not covered by
$\Upsilon_s(K_{s_0})$ lies in the boundary strip
$\operatorname{dist}(x,\partial\Omega_s)<\delta_0$. On this region
strict convexity follows from \eqref{eq:uniform-boundary-strip}.
Thus $D^2v_s>0$ throughout $\Omega_s$, and
$\mathcal I_{\mathrm{cvx}}$ is open.

It remains to prove closedness. Let
$s_j\in\mathcal I_{\mathrm{cvx}}$ and suppose that $s_j\to s_0$.
For every compact set $\Omega'\Subset\Omega_{s_0}$,
Lemma~\ref{lem:real-local-stability} gives
\[
D^2v_{s_j}(\Upsilon_{s_j}(x))
\longrightarrow
D^2v_{s_0}(x)
\]
uniformly on $\Omega'$. Since
$D^2v_{s_j}>0$, we obtain
\begin{equation}\label{eq:weak-convex-limit}
D^2v_{s_0}\ge0
\qquad\text{in }\Omega_{s_0}.
\end{equation}

We may now apply
Theorem~\ref{thm:realn} to $v_{s_0}$.
It follows from \eqref{eq:weak-convex-limit} that
$\operatorname{rank}D^2v_{s_0}$ is constant in the connected domain
$\Omega_{s_0}$. On the other hand,
\eqref{eq:uniform-boundary-strip} shows that $D^2v_{s_0}$ is positive
definite on a nonempty open subset of $\Omega_{s_0}$. Hence its
constant rank is $n$, and therefore
\[
D^2v_{s_0}>0
\qquad\text{throughout }\Omega_{s_0}.
\]
Thus $s_0\in\mathcal I_{\mathrm{cvx}}$, and
$\mathcal I_{\mathrm{cvx}}$ is closed.

Since $[0,1]$ is connected and
$\mathcal I_{\mathrm{cvx}}$ is nonempty, open, and closed, we conclude
that
\[
\mathcal I_{\mathrm{cvx}}=[0,1].
\]
In particular,
\[
D^2v_1>0
\qquad\text{in }\Omega.
\]

Finally, $u_1$ is the admissible solution of \eqref{eq:real-main}
with zero Dirichlet data on $\Omega$. By the uniqueness statement in
Proposition~\ref{prop:real-solvability}, it coincides with the
solution $u$ appearing in
Theorem~\ref{thm:real-strict-convexity}. In particular, $u<0$ in
$\Omega$ and $v_1=-\sqrt{-u}$. Consequently,
\[
D^2(-\sqrt{-u})>0
\qquad\text{in }\Omega.
\]
This proves Theorem~\ref{thm:real-strict-convexity}.



\section{Optimality of the square-root exponent}
\label{sec:square-root-optimality}

The strict-convexity result proved above is equivalent to the
strict concavity of $\sqrt{-u}$. It is therefore natural to ask
whether the exponent $1/2$ can be replaced, uniformly over all
smooth uniformly convex domains and admissible solutions, by a
larger exponent. We show that this is impossible.

Moreover, if $0<\gamma\le \frac12$, then
\[
(-u)^\gamma
=
\bigl(\sqrt{-u}\bigr)^{2\gamma}
\]
is concave, since $t\mapsto t^{2\gamma}$ is increasing and concave
on $(0,\infty)$. Thus it remains only to rule out every
$\gamma>\frac12$.

The counterexample is based on a scaling mechanism near a suitably
chosen boundary point. The solution can be forced to vanish at most
quadratically, whereas a nonnegative concave function vanishing at
one endpoint of a segment must grow at least linearly along that
segment. Thus the concavity of $(-u)^\gamma$ would require
$2\gamma\le1$.


We now prove sharpness for the real graphical
$\sigma_2$-curvature equation. The construction begins with a
constant-right-hand-side $\sigma_2$-Hessian equation on domains
approaching the vertex of a convex cone. A small perturbation then
introduces the graphical $T_2$ term, and a final spatial scaling
recovers exactly equation~\eqref{eq:real-main}.

\begin{theorem}
\label{thm:real-square-root-optimality}
Let $n\ge3$. For every $\gamma>1/2$, there exist a bounded, smooth,
uniformly strictly convex domain
$\Omega_\gamma\subset\mathbb R^n$ and a smooth graphical
$\Gamma_2$-admissible solution $u_\gamma<0$ satisfying
\[
\sigma_2(\kappa[u_\gamma])
=
(1+|Du_\gamma|^2)^{-2}
\quad\text{in }\Omega_\gamma,
\qquad
u_\gamma=0
\quad\text{on }\partial\Omega_\gamma,
\]
such that $(-u_\gamma)^\gamma$ is not concave in $\Omega_\gamma$.
Consequently, the exponent $1/2$ in
Theorem~\ref{thm:real-strict-convexity} is optimal among universal
power-concavity exponents.
\end{theorem}

\begin{proof}
Fix $\gamma>1/2$. We divide the construction into three steps.

\medskip
\noindent\textbf{Step 1. A sharp counterexample for the limiting
$\sigma_2$-Hessian equation.}

Set
\[
\tau_n^2=\frac{n-2}{4},
\qquad
\mathfrak c_n=\frac{1}{\sqrt{(n-1)(n-2)}}.
\]
Write $x=(x',x_n)\in\mathbb R^{n-1}\times\mathbb R$ and define the
open convex cone
\begin{equation}
\label{eq:real-sharp-cone}
\mathfrak C_n
=
\left\{
x_n>0,\quad |x'|<\tau_n x_n
\right\}.
\end{equation}
Consider the quadratic function
\[
\mathfrak b_{\mathrm{cone}}(x)
=
\mathfrak c_n
\left(
|x'|^2-\tau_n^2x_n^2
\right).
\]
Then $\mathfrak b_{\mathrm{cone}}<0$ in $\mathfrak C_n$, while
\[
D^2\mathfrak b_{\mathrm{cone}}
=
2\mathfrak c_n
\operatorname{diag}
(1,\ldots,1,-\tau_n^2).
\]
A direct computation gives
\[
\sigma_1(D^2\mathfrak b_{\mathrm{cone}})>0,
\qquad
\sigma_2(D^2\mathfrak b_{\mathrm{cone}})
=
4\mathfrak c_n^2
\left[
\binom{n-1}{2}-(n-1)\tau_n^2
\right]
=1.
\]
Hence
\begin{equation}
\label{eq:real-sharp-cone-solution}
D^2\mathfrak b_{\mathrm{cone}}\in\Gamma_2,
\qquad
\sigma_2(D^2\mathfrak b_{\mathrm{cone}})=1.
\end{equation}
Notice that $D^2\mathfrak b_{\mathrm{cone}}$ is not positive
definite; the construction requires only $\Gamma_2$-admissibility.

Since $e_n\in\mathfrak C_n$, choose $r_{\mathrm c}>0$ such that
$B_{r_{\mathrm c}}(e_n)\Subset\mathfrak C_n$. Let
$d_j\downarrow0$. By smoothing the convex hull
\[
\operatorname{conv}
\left(
\{d_je_n\}\cup B_{r_{\mathrm c}}(e_n)
\right)
\]
inside $\mathfrak C_n$, we obtain bounded smooth uniformly strictly
convex domains $\mathcal O_j$ such that
each $\mathcal O_j$ is uniformly strictly convex, with the convexity constant allowed to depend on $j$, and
\begin{equation}
\label{eq:real-sharp-domains}
B_{r_{\mathrm c}/2}(e_n)
\Subset
\mathcal O_j
\Subset
\mathfrak C_n,
\end{equation}
and boundary points $y_j\in\partial\mathcal O_j$ such that
$y_j\to0$. No uniform lower bound for the boundary curvatures of the
whole sequence is required.

On $\mathcal O_j$, let $\psi_j$ be the unique smooth admissible
solution of
\begin{equation}
\label{eq:real-sharp-hessian}
\begin{cases}
\sigma_2(D^2\psi_j)=1,
& x\in\mathcal O_j,\\
\psi_j=0,
& x\in\partial\mathcal O_j.
\end{cases}
\end{equation}
Existence and uniqueness follow from the classical
$\sigma_2$-Hessian Dirichlet theory \cite{CNS85}. Since
$D^2\psi_j\in\Gamma_2$, one has $\Delta\psi_j>0$, and hence
$\psi_j<0$ in $\mathcal O_j$.

Because $\mathcal O_j\Subset\mathfrak C_n$,
$\mathfrak b_{\mathrm{cone}}<0=\psi_j$ on
$\partial\mathcal O_j$. Together with
\eqref{eq:real-sharp-cone-solution} and
\eqref{eq:real-sharp-hessian}, the comparison principle gives
\[
\mathfrak b_{\mathrm{cone}}
\le
\psi_j
\le0
\qquad\text{in }\mathcal O_j.
\]
Therefore
\begin{equation}
\label{eq:real-sharp-quadratic-upper}
-\psi_j(x)
\le
-\mathfrak b_{\mathrm{cone}}(x)
\le
\mathfrak c_n\tau_n^2x_n^2.
\end{equation}

We next obtain a lower bound at the fixed interior point $e_n$.
On $B_{r_{\mathrm c}/2}(e_n)$ define
\[
\mathfrak b_{\mathrm{ball}}(x)
=
\frac{
|x-e_n|^2-(r_{\mathrm c}/2)^2
}{
2\sqrt{\binom n2}
}.
\]
Then
\[
\sigma_2(D^2\mathfrak b_{\mathrm{ball}})=1,
\qquad
\mathfrak b_{\mathrm{ball}}=0
\quad\text{on }
\partial B_{r_{\mathrm c}/2}(e_n).
\]
Since the ball is compactly contained in $\mathcal O_j$, we have
$\psi_j<0=\mathfrak b_{\mathrm{ball}}$ on its boundary. Another
application of the comparison principle gives
$\psi_j\le\mathfrak b_{\mathrm{ball}}$ in the ball. In particular,
\begin{equation}
\label{eq:real-sharp-interior-lower}
-\psi_j(e_n)
\ge
\frac{r_{\mathrm c}^2}
{8\sqrt{\binom n2}}
=:c_{\mathrm{ball}}>0,
\end{equation}
where $c_{\mathrm{ball}}$ is independent of $j$.

Choose a number $0<\theta<1$ and set
$x_j=(1-\theta)y_j+\theta e_n$. Then
$x_j\in\mathcal O_j$ and $x_j\to\theta e_n$. By
\eqref{eq:real-sharp-quadratic-upper},
\[
\limsup_{j\to\infty}
(-\psi_j(x_j))
\le
\mathfrak c_n\tau_n^2\theta^2.
\]
Hence, for all sufficiently large $j$,
\[
-\psi_j(x_j)
\le
2\mathfrak c_n\tau_n^2\theta^2.
\]
Since $2\gamma-1>0$, we can first choose $\theta>0$ sufficiently
small that
\begin{equation}
\label{eq:real-sharp-scale-gap}
\left(
2\mathfrak c_n\tau_n^2\theta^2
\right)^\gamma
<
\theta c_{\mathrm{ball}}^\gamma.
\end{equation}
We then fix an index $j_0$ sufficiently large so that
\[
(-\psi_{j_0}(x_{j_0}))^\gamma
<
\theta(-\psi_{j_0}(e_n))^\gamma.
\]
Since $\psi_{j_0}(y_{j_0})=0$ and
$x_{j_0}=(1-\theta)y_{j_0}+\theta e_n$, this is the strict reverse
of the concavity inequality for $(-\psi_{j_0})^\gamma$.

From now on set
\[
\mathcal O_0=\mathcal O_{j_0},
\qquad
\psi_0=\psi_{j_0},
\qquad
y_0=y_{j_0},
\qquad
x_0=x_{j_0}.
\]
Thus
\begin{equation}
\label{eq:real-sharp-strict-gap}
(-\psi_0(x_0))^\gamma
<
\theta(-\psi_0(e_n))^\gamma.
\end{equation}

\medskip
\noindent\textbf{Step 2. Perturbation to the graphical
fixed-right-hand-side equation.}

On the fixed domain $\mathcal O_0$, consider for
$\varepsilon\ge0$ the Dirichlet problem
\begin{equation}
\label{eq:real-sharp-perturb}
\begin{cases}
\sigma_2(D^2\psi_\varepsilon)
+
\varepsilon
T_2(D^2\psi_\varepsilon)
[D\psi_\varepsilon,D\psi_\varepsilon]
=1,
& x\in\mathcal O_0,\\
\psi_\varepsilon=0,
& x\in\partial\mathcal O_0.
\end{cases}
\end{equation}
At $\varepsilon=0$ the solution is $\psi_0$.

Define
\[
\mathcal J(\varepsilon,f)
=
\sigma_2(D^2f)
+
\varepsilon T_2(D^2f)[Df,Df]
-1.
\]
The linearization with respect to $f$ at $(0,\psi_0)$ is
\[
D_f\mathcal J(0,\psi_0)[\varphi]
=
T_1(D^2\psi_0)^{ij}\varphi_{ij}.
\]
Since $D^2\psi_0\in\Gamma_2$ on
$\overline{\mathcal O_0}$, the coefficient matrix
$T_1(D^2\psi_0)$ is uniformly positive definite. The maximum
principle, Schauder theory, and the Fredholm alternative show that
the zero-Dirichlet linearization is invertible between the standard
Hölder spaces. The implicit function theorem therefore gives
$\varepsilon_0>0$ and a smooth local branch satisfying
\begin{equation}
\label{eq:real-sharp-IFT}
\psi_\varepsilon
\longrightarrow
\psi_0
\quad\text{in }C^{2,\alpha}(\overline{\mathcal O_0})
\quad\text{as }\varepsilon\downarrow0
\end{equation}
for any fixed Hölder exponent $0<\alpha<1$.

After decreasing $\varepsilon_0$ if necessary,
$D^2\psi_\varepsilon\in\Gamma_2$ and
$\psi_\varepsilon<0$ for $0\le\varepsilon<\varepsilon_0$.
Moreover, the strict finite-point inequality
\eqref{eq:real-sharp-strict-gap} is stable under uniform convergence.
Hence
\begin{equation}
\label{eq:real-sharp-gap-perturbed}
(-\psi_\varepsilon(x_0))^\gamma
<
\theta(-\psi_\varepsilon(e_n))^\gamma
\end{equation}
for every sufficiently small $\varepsilon>0$.

\medskip
\noindent\textbf{Step 3. Scaling back to the original graphical
equation.}

Fix a sufficiently small $\varepsilon>0$ for which
\eqref{eq:real-sharp-gap-perturbed} holds, and define
\[
\Omega_\gamma
=
\sqrt{\varepsilon}\,\mathcal O_0,
\qquad
u_\gamma(x)
=
\varepsilon
\psi_\varepsilon
\left(
\frac{x}{\sqrt{\varepsilon}}
\right),
\quad x\in\Omega_\gamma.
\]
Then
\[
Du_\gamma(x)
=
\sqrt{\varepsilon}\,
D\psi_\varepsilon
\left(
\frac{x}{\sqrt{\varepsilon}}
\right),
\qquad
D^2u_\gamma(x)
=
D^2\psi_\varepsilon
\left(
\frac{x}{\sqrt{\varepsilon}}
\right).
\]
Using \eqref{eq:real-sharp-perturb}, we obtain
\[
\sigma_2(D^2u_\gamma)
+
T_2(D^2u_\gamma)[Du_\gamma,Du_\gamma]
=1.
\]
By the fixed-right-hand-side identity established in
Subsection~\ref{subsubsec:real-ball}, this is exactly
\[
\sigma_2(\kappa[u_\gamma])
=
(1+|Du_\gamma|^2)^{-2}.
\]
Also $u_\gamma=0$ on $\partial\Omega_\gamma$ and
$u_\gamma<0$ in $\Omega_\gamma$.

We finally verify graphical admissibility. By
\eqref{eq:real-sharp-IFT},
$D^2\psi_\varepsilon\to D^2\psi_0$ uniformly, while
$Du_\gamma=O(\sqrt{\varepsilon})$. Hence the symmetric Weingarten
representative
\[
\frac{1}{\sqrt{1+|Du_\gamma|^2}}
(I+Du_\gamma\otimes Du_\gamma)^{-1/2}
D^2u_\gamma
(I+Du_\gamma\otimes Du_\gamma)^{-1/2}
\]
converges uniformly, as $\varepsilon\downarrow0$, to
$D^2\psi_0$. Since the latter takes values in a compact subset of
$\Gamma_2$, the graph of $u_\gamma$ is $\Gamma_2$-admissible when
$\varepsilon$ is sufficiently small.

Set
\[
\widehat y=\sqrt{\varepsilon}\,y_0,
\qquad
\widehat x=\sqrt{\varepsilon}\,x_0,
\qquad
\widehat e=\sqrt{\varepsilon}\,e_n.
\]
Then
$\widehat x=(1-\theta)\widehat y+\theta\widehat e$ and
$u_\gamma(\widehat y)=0$. By
\eqref{eq:real-sharp-gap-perturbed},
\[
(-u_\gamma(\widehat x))^\gamma
=
\varepsilon^\gamma
(-\psi_\varepsilon(x_0))^\gamma
<
\theta\varepsilon^\gamma
(-\psi_\varepsilon(e_n))^\gamma
=
\theta(-u_\gamma(\widehat e))^\gamma.
\]
Therefore
\[
(-u_\gamma(\widehat x))^\gamma
<
(1-\theta)(-u_\gamma(\widehat y))^\gamma
+
\theta(-u_\gamma(\widehat e))^\gamma,
\]
which is the strict reverse of the concavity inequality.
Thus $(-u_\gamma)^\gamma$ is not concave.

Since $\gamma>1/2$ was arbitrary,
Theorem~\ref{thm:real-square-root-optimality} follows.
\end{proof}

\begin{remark}
\label{rem:real-square-root-scaling}
The origin of the exponent $1/2$ is the scale balance in
\eqref{eq:real-sharp-scale-gap}. Near the cone vertex the comparison
argument gives $-u\lesssim\theta^2$, whereas concavity of
$(-u)^\gamma$ along a segment issuing from the boundary would force
$(-u)^\gamma\gtrsim\theta$. Compatibility as $\theta\downarrow0$
therefore requires $2\gamma\le1$.
\end{remark}

\small
\bibliographystyle{alpha}
\bibliography{reference}
\medskip
\noindent
(Shuning Xu) Department of Mathematics, University of Science and Technology of China, Hefei, 230026, Anhui Province, China.\;
Email address: xushuning@mail.ustc.edu.cn
\end{document}